\documentclass[11pt]{amsart}
\usepackage{amssymb, amsthm, amsfonts, latexsym,amsmath, amscd,hyperref}

\theoremstyle{plain}
\newtheorem {theorem} {Theorem}[section]
\newtheorem{prop} [theorem]{Proposition}
\newtheorem{lemma}[theorem] {Lemma}
\newtheorem{corr}[theorem]{Corollary}
\theoremstyle{definition}
\newtheorem{exam}[theorem]{Example}
\newtheorem{remark}[theorem]{Remark}
\newtheorem{defin}[theorem]{Definition}
\numberwithin{equation}{section}
\newcommand{\calr}{{\mathcal R}}

\newcommand{\eqman}{ {\mathcal R}^\infty}

\newcommand{\bet}{\text{\tiny$\beta$}}

\newcommand{\hook} {\mathbin{\raise2.5pt\hbox{\hbox{{\vbox{\hrule height.4pt
width6pt depth0pt}}}\vrule height3pt width.4pt depth0pt}\,}}

\begin{document}

%
%

\title[The inverse problem of the calculus of variations]
{The inverse problem of the calculus of variations for
systems of second-order partial differential equations in the plane}

\author[M. Biesecker]{Matt Biesecker}

\address{Department of Mathematics and Statistics \\ South Dakota State University \\ Brookings, SD 57006}

\email{Matt.Biesecker@sdstate.edu}

\date{September 20, 2009}



\begin{abstract}  A complete solution to the multiplier version of the inverse
problem of the calculus of variations is given for a class of
hyperbolic systems of second-order partial differential equations
in two independent variables.   The necessary and sufficient
algebraic and differential conditions for the existence of a
variational multiplier are derived.  It is shown that the number
of independent variational multipliers is  determined by the
nullity of a completely algebraic system of equations associated
to the given system of partial differential equations.  An
algorithm for solving the inverse problem is demonstrated on
several examples. Systems of second-order partial differential
equations in two independent and dependent variables are studied
and systems which have more than one variational formulation are
classified up to contact equivalence.
\end{abstract}
\maketitle

%
%

\section{Introduction}  In this paper we examine the inverse
problem of the calculus of variations for systems of partial
differential equations
\begin{equation}\label{mainsys}
u^\alpha_{xy} = f^\alpha(x,y,u^\gamma,u^\gamma_x,u^\gamma_y),
\qquad \alpha,\, \gamma = 1,2,\dots,m,
\end{equation}
in $m$ dependent variables $u^\alpha$ and two independent
variables $x$ and $y.$  Systems \eqref{mainsys} arise in various
contexts such as symmetry reduction in general relativity,
non-linear $\sigma$-models, and generalized Toda lattices (see
\cite{Forger1981},\cite{Ernie2001},\cite{LeznovSaveliev1980},\cite{Plebanski1988},
and \cite{Vass1994}).   For systems \eqref{mainsys} we give a
complete solution to the variational multiplier version of the
inverse problem.  In particular, we give an explicit algorithm for
determining the number of possible inequivalent Lagrangians for a
given system \eqref{mainsys}.

The \emph{variational multiplier problem} is stated as follows:
Given a system of differential equations
$F_{\alpha}(x^i,u^\gamma,u^\gamma_i,u^\gamma_{ij},
u^\gamma_{ijk},\dots)=0,$ does there exist a Lagrangian
$L(x^i,u^\gamma,u^\gamma_i,u^\gamma_{ij},\dots)$ and functions
$M^\gamma_\alpha,$ with $\det \left(M^\gamma_\alpha\right) \neq
0,$ such that
\begin{equation}\label{multproperty}
E_{\alpha}  (L)  =  M^\gamma_\alpha  F_\gamma,
\end{equation}
where $E$ is the Euler-Lagrange operator.  If such an
$M^\alpha_\gamma$ exists, then it is referred to as a {\it
variational multiplier.} If \eqref{multproperty} is satisfied, the
equations $F_{\alpha}=0$ are equivalent to a system of
Euler-Lagrange equations in the sense that solutions of the
Euler-Lagrange equations for $L$ are solutions to $F_{\alpha}=0$
and conversely solutions to $F_\alpha=0$ are solutions for
$E_{\alpha}(L)=0.$

Also of importance is the number of Lagrangians for a particular
system.   It is well known that two Lagrangians $L$ and $L^{'}$
have identical Euler-Lagrange expressions if and only if $L$ and
$L^{'}$ differ  by a total divergence, at least locally. Therefore
we consider two Lagrangians to be {\it equivalent} if they differ
by a total divergence.    We note that it is possible to have a
system of differential equations which have two or more
inequivalent Lagrangians (and thus more than one multiplier).

The variational multiplier problem is of some interest in
theoretical physics. It is widely accepted that fundamental
physical theories can be derived from an action principle, or
Lagrangian.    For a given theory it is important to determine
whether the action is unique.     Examples with multiple
Lagrangians exist in Newtonian Mechanics  and  the $SU(2)$ chiral
model (\cite{Henneaux1984}, \cite{HenneauxShep1982}.)

The simplest case of the variational multiplier problem is for a
scalar second-order ordinary differential equation
$u_{xx}=F(x,u,u_x).$ It has been shown by several authors,
including Darboux \cite{Darboux1894}, that any scalar second-order
ODE is variational and the most general multiplier (and
Lagrangian) depends on an arbitrary function of two variables. The
multiplier problem for systems of second-order ordinary
differential equations  has been studied by many authors including
Douglas \cite{Douglas1941}, Anderson and Thompson
\cite{AndThomp1992},  Thompson, Crimpin, Sarlet, Prince and
Martinez (See \cite{CPST1999},\cite{CSMP1994},\cite{SCM1998}, and
\cite{STG2002}.) Our solution to the multiplier problem for
systems \eqref{mainsys} is partially based on ideas developed in
Anderson and Thompson's paper \cite{AndThomp1992}. In that paper
Anderson and Thompson used the variational bicomplex to derive a
system of algebraic and differential conditions, with components
of the multiplier matrix as unknowns, for the existence of a
multiplier.   They showed that there is a one-to-one
correspondence between variational multipliers for the given
system and certain cohomology classes in the variational bicomplex
associated to the given system of differential equations. While a
significant amount of research has been done on the variational
multiplier problem for second-order ODE systems, a complete
solution remains elusive in the sense that there is no general
closed form characterization, in terms of invariants for the
system, for determining the existence and degree of uniqueness of
multipliers for the given system.

The variational multiplier problem for higher order scalar
ordinary differential equations has been studied by Fels
\cite{Fels1996} and Jur\'a\v{s} \cite{Juras2000}.  Fels
\cite{Fels1996} obtained a complete solution to the variational
multiplier problem for fourth-order scalar ordinary differential
equations $u_{xxxx} = F(x,u,u_x,u_{xx},u_{xxx} ).$ Using Cartan's
method of equivalence, he was able to produce two differential
invariants whose vanishing completely characterizes the existence
of a variational multiplier. Unlike the second-order case, the
multiplier is unique up to a constant multiple.   In
\cite{Juras2000}, Jur\'a\v{s}  obtained a similar solution for
sixth and eighth-order scalar ordinary differential equations,
although the differential invariants are increasingly complicated
for higher order systems.

For partial differential equations, there are fewer papers on the
variational multiplier problem.  Anderson and Duchamp
\cite{AndDuchamp1984} studied the variational multiplier problem
for scalar second-order quasilinear partial differential equations
$F= 0$ where $F=  A^{ij}(x^k, u, u_k )  u_{ij} + B(x^k, u, u_k ).$
Anderson and Duchamp \cite{AndDuchamp1984}  proved that if $
\det(A^{ij}) \neq 0,$  then $F$ has a variational multiplier if
and only if a certain 1-form $\chi$ is closed. Moreover, $\chi$ is
expressed explicitly in terms of $F$ and its derivatives and the
multiplier, if it exists,  is unique up to a constant multiple.
They also  showed that second-order scalar evolution equations are
never variational. Jur\'a\v s \cite{Juras1997} examined the
inverse problem for a scalar hyperbolic second-order partial
differential equations in two independent variables.    He was
able to show that an equation is variational if and only if two
particular differential invariants $H$ and $K$ are identically
equal. Moreover, the multiplier is unique up to a constant.
Jur\'a\v s' result is equivalent to the Anderson-Duchamp result if
the equation is quasilinear. However, his results also apply to
hyperbolic Monge-Ampere equations.

In this paper we study the variational multiplier problem for $f$-Gordon
systems
\begin{equation}\label{gordy1}
u_{xy}^\alpha = f^\alpha(x,y,u^\gamma, u^\gamma_x, u^\gamma_y),
\qquad \alpha, \, \gamma=1,2, \dots , m.
\end{equation}
We remark that we are working under the assumption that the
functions $f^\alpha$ are $C^\infty$ on some open set $U \subset
{\mathbf R^2} \times{\mathbf R}^{3m}.$ In the following section we
outline the complete solution to the multiplier version of the
inverse problem for systems \eqref{gordy1}. In particular, we give
an algorithm for determining the number of variation multipliers
(and Lagrangians) for a given system \eqref{gordy1}. We show that
the most general multiplier depends on finitely many constants
determined by the dimension of the nullspace of a certain matrix
depending on the functions $f^\alpha$ and their derivatives. In
Section 3 we demonstrate our algorithm for solving the inverse
problem on several examples. In Section 4 we classify all systems
\eqref{gordy1} in two dependent variables that admit two or more
inequivalent Lagrangians.  In Sections 5 and 6, we turn our
attention to proving two technical propositions stated in Section
2. In particular, we define the \emph{variational bicomplex} for
systems \eqref{gordy1} and show there is a one-to-one
correspondence between the Lagrangians for systems \eqref{gordy1}
and special classes of $3$-forms $\omega$ with $d \omega = 0.$  We
then derive necessary and sufficient algebraic and differential
conditions for the existence of a variational multiplier.  In the
appendix we prove a general theorem on the existence of solutions
to systems of combined differential and algebraic equations that
allows us to determine the number of variational multipliers from
purely algebraic data.

This work is an extension of a result established in my PhD dissertation.  I wish to thank my advisor Ian
Anderson for suggesting this problem and an uncountable number of helpful discussions on this subject.

%
%

\section{Main Results}   In this section we state our solution to
the variational multiplier problem for systems \eqref{mainsys}.
Our solution depends on two propositions which we will prove in
Section 6. The first proposition gives us a general normal form
for variational systems \eqref{mainsys} and states that the
Lagrangian associated to a particular variational multiplier is
unique up to modification by a total divergence.

\begin{prop}\label{LnormalF}  If there exists a first-order Lagrangian
$L(x,y,u^\gamma,u^\gamma_x,u^\gamma_y)$ and a non-degenerate
variational multiplier
$M_{\alpha\bet}(x,y,u^\gamma,u^\gamma_x,u^\gamma_y)$ such that
\begin{equation*}
E_\alpha(L) = M_{\alpha\bet} \left( u^\bet_{xy} -
f^\bet(x,y,u^\gamma,u^\gamma_x,u^\gamma_y) \right),
\end{equation*}
then $M_{\alpha\bet}=M_{\alpha\bet}(x,y,u^\gamma),$
$M_{\alpha\bet} = M_{\bet\alpha},$
\begin{equation}
\dfrac{\partial^2 f^\alpha}{\partial u^\bet_x \partial u^\gamma_x} = 0, \qquad
\dfrac{\partial^2 f^\alpha}{\partial u^\bet_y \partial u^\gamma_y} = 0, \quad
\forall \, \alpha,\beta,\gamma=1,2,\dots,m.
\end{equation}
Moreover, if $\tilde L(x,y,u^\gamma,u^\gamma_x,u^\gamma_y)$ is
another Lagrangian associated to the variational multiplier
$M_{\alpha\bet}$ such that $E_\alpha(\tilde L) = M_{\alpha\bet} \,
(u^\bet_{xy} - f^\bet),$ then $ \tilde L = L + \text{Div} \
{\mathbf Q},$ where ${\mathbf Q}=( \, Q_1(x,y,u^\gamma),
Q_2(x,y,u^\gamma) \, )$ and  $\text{Div} \ {\mathbf Q} = D_x Q_1 +
D_y Q_2.$
\end{prop}
It follows immediately from Proposition \eqref{LnormalF} that we
may restrict ourselves to systems of the form
\begin{equation} \label{normF}
u^\alpha_{xy} + C^\alpha_{\bet\gamma}(x,y,u^\tau) \, u^\bet_x \, u^\gamma_y +
A^\alpha_\gamma (x,y,u^\tau) u^\gamma_x + B^\alpha_\gamma(x,y,u^\tau)
u^\gamma_y + E^\alpha(x,y,u^\tau)=0.
\end{equation}
We remark that the condition that $M_{\alpha\bet}$ is symmetric
and has no first-order derivative dependence also follows from a
result established by Henneaux  in \cite{Henneaux1984}.

Of paramount importance in our study of the inverse problem are
three quantities $H^\gamma_\alpha,$ $K^\gamma_\alpha,$ and
$S^\gamma_{\alpha\bet}$ defined by
\begin{subequations} \label{hks}
\begin{gather} H^\gamma_\alpha ={\small \text{$
  \dfrac{\partial f^\gamma} {\partial u^\alpha} +
\dfrac{\partial f^\gamma}{\partial u^\sigma_y} \, \dfrac{\partial
f^\sigma}{\partial u^\alpha_x} -  \dfrac{\partial^2
f^\gamma}{\partial u^\alpha_x \partial x}  - \dfrac{\partial^2
f^\gamma}{\partial u^\alpha_x \partial u^\bet} \, u^\bet_x -
\dfrac{\partial^2 f^\gamma}{\partial u^\alpha_x \partial u^\bet_y}
\, f^\bet - \dfrac{\partial^2 f^\gamma}{\partial u^\alpha_x
\partial u_x^\bet} \, u^\bet_{xx} \, , $ } }
\\
K^\gamma_\alpha = {\small\text{$\ \dfrac{\partial f^\gamma} {\partial u^\alpha}
+ \dfrac{\partial f^\gamma}{\partial u^\sigma_x} \, \dfrac{\partial
f^\sigma}{\partial u^\alpha_y} - \dfrac{\partial^2 f^\gamma}{\partial
u^\alpha_y \partial y}  - \dfrac{\partial^2 f^\gamma}{\partial u^\alpha_y
\partial u^\bet} \, u^\bet_y - \dfrac{\partial^2 f^\gamma}{\partial u^\alpha_y
\partial u^\bet_x} \, f^\bet - \dfrac{\partial^2
f^\gamma}{\partial u^\alpha_y
\partial u_y^\bet} \, u^\bet_{yy}\, ,$}}
\\
S^\alpha_{\bet\gamma} = \text{\small\text{$ \dfrac{\partial^2
f^\gamma}{\partial u^\bet_x \partial u^\alpha_y} - \dfrac{\partial^2
f^\gamma}{\partial u^\alpha_x
\partial u^\bet_y}. $ } }
\end{gather}
\end{subequations}
In the formulas \eqref{hks} and throughout this paper we are
adopting the Einstein summation convention where repeated indices
are summed upon.  We remark that the quantities $H^\gamma_\alpha,$
$K^\gamma_\alpha,$ and $S^\gamma_{\alpha\bet}$ are relative
invariants under the pseudo-group of local contact transformations
that preserve systems \eqref{mainsys}. The functions
$H^\alpha_\gamma$ and $K^\alpha_\gamma$ are natural extensions of
the \emph{generalized Laplace invariants} defined in
\cite{AndKam1997} for scalar hyperbolic equations.   For systems
\eqref{normF} it follows from \eqref{hks} that the functions
$H^\gamma_\alpha$ and $K^\gamma_\alpha$ have no second-order
derivative dependence and $S^\gamma_{\alpha\bet}$ has no
first-order derivative dependence, that is $H^\gamma_\alpha =
H^\gamma_\alpha(x,y,u^\tau,u^\tau_x,u^\tau_y),$ $K^\gamma_\alpha =
K^\gamma_\alpha(x,y,u^\tau,u^\tau_x,u^\tau_y),$ and
$S_{\alpha\bet} = S^\gamma_{\alpha\bet}(x,y,u^\tau),$ where
$H^\gamma_\alpha$ and $K^\gamma_\alpha$ are quadratic in the
variables $u^\bet_x$ and $u^\bet_y.$ Our second proposition
characterizes the algebraic and differential conditions on a
variational multiplier for a system \eqref{normF}.

\begin{prop}\label{PropMcond}  A system \eqref{normF} is multiplier
variational with a multiplier $M_{\alpha\bet}$ if and only if
$M_{\alpha\bet}$ satisfies $\det M_{\alpha\bet} \neq 0$ and
\begin{subequations}\label{EMcon}
\begin{align} \label{MHMK}
M_{\alpha\sigma} H^\sigma_\bet &= M_{\bet\sigma} K^\sigma_\alpha,
\qquad M_{\alpha\sigma} S^\sigma_{\bet\gamma} = -M_{\bet\sigma}
S^\sigma_{\alpha\gamma},
\\ \label{dMomega}
dM_{\alpha\bet} &= M_{\alpha\sigma} \, \Omega^\sigma_\bet + M_{\bet\sigma} \,
\Omega^\sigma_{\alpha}, \
\end{align}
where $\Omega^\sigma_\alpha = C^\sigma_{\alpha\tau} \, du^\tau +
A^\sigma_\alpha \, dy + B^\sigma_\alpha dx.$
\end{subequations}
\end{prop}

\begin{remark}\label{RMcon}
In Section 6, we will derive the algebraic and differential
conditions \eqref{EMcon} and give an algorithm for constructing
the first-order Lagrangian $L$ associated to a variational
multiplier $M_{\alpha\bet}$ satisfying \eqref{EMcon}.   The
algebraic conditions \eqref{MHMK} include the integrability
conditions for $d^2 M_{\alpha\bet}=0$ for \eqref{dMomega}
\emph{and} genuine algebraic constraints not arising from the
integrability conditions.   We will show that the set of all
solutions to \eqref{EMcon} is a finite dimensional real vector
space with the dimension determined by the nullity of a completely
algebraic system of equations.  In particular, for any system of
the form \eqref{normF}, the solutions to \eqref{EMcon} are
completely determined. However,  in some cases there are
non-trivial solutions $M_{\alpha\bet}$ to \eqref{EMcon} that do
not satisfy $\det M_{\alpha\bet} \neq 0.$ In practice, we first
determine a basis for the solutions to \eqref{EMcon} and then
determine if a non-degenerate variational multiplier can be
constructed from a linear combination of the basis solutions.
\end{remark}

We are now ready to state our solution to the inverse problem
which says that the variational multipliers $M_{\alpha\bet}$
satisfying \eqref{EMcon} are completely characterized by the
solutions to an algebraic system $\Phi^{\alpha\bet}_a \,
M_{\alpha\bet} =0,$ where $\Phi$ depends on the given system
$u^\alpha_{xy} = f^\alpha.$

\begin{theorem}\label{inversesolution} Let $u^\alpha_{xy}= f^\alpha(
x,y,u^\gamma,u^\gamma_x, u^\gamma_y)$ be a system of $m$ partial
differential equations where $f^\alpha \in C^\infty(U)$ for some
open set $U \subset {\mathbf R^2} \times {\mathbf R^{3m}}.$   Then
there is a matrix ${\mathbf \Phi},$ with $ \text{Rank}({\mathbf
\Phi}) \leq m(m+1)/2,$ depending the functions $H^\gamma_\alpha,$
$K^\gamma_\alpha,$ and $S^\gamma_{\alpha\bet}$ and their
derivatives, whose nullspace completely determines the number of
linearly independent solutions to \eqref{EMcon}. Specifically, if
$r=\text{Rank} \ \Phi(z) $ is constant for all points $z$ in a
open set $V \subset U,$ then at every point $z_0 \in V$ there is a
neighborhood $W \subset V$ of $z_0$ where the set of solutions to
\eqref{EMcon} is a $m(m+1)/2 - r$ dimensional vector space over
${\mathbf R}.$
\end{theorem}

\begin{proof}
The algebraic system for a multiplier $M_{\alpha\bet}$ is
constructed as follows.  As a consequence of proposition
\eqref{LnormalF}, the functions $M_{\alpha\bet}$ have no 1-jet
dependence and $H^\gamma_\alpha$ and $K^\gamma_\alpha$ are
quadratic in $u^\epsilon_x$ and $u^\tau_y.$  We can decompose the
algebraic condition $M_{\alpha\gamma} H^\gamma_\bet=M_{\bet\gamma}
K^\gamma_\alpha$  into several linear algebraic systems on the
components of the multiplier $M_{\alpha\bet},$ with the
coefficients depending only on $x,y,$ and $u.$   We then express
all algebraic conditions \eqref{MHMK} on $M_{\alpha\gamma}$ as a
single system of linear equations
\begin{equation}\label{initalg}
\left(\Phi_{0}\right)_a^{\alpha\bet} (x,y,u^\tau) \,
M_{\alpha\bet}  = 0, \qquad a=1,2,\dots,k_0,
\end{equation}
where $\Phi_0$ may be viewed as $k_0 \times m(m+1)/2$ matrix.
Differentiating \eqref{initalg} and substituting from
\eqref{dMomega}, we get the algebraic condition
\begin{equation}\label{dphi}
\left[ \, \left(d \Phi_0\right)^{\alpha\bet}_a +
\left(\Phi_0\right)^{\alpha\sigma}_a \, \Omega^\bet_\sigma +
\left(\Phi_0\right)^{\sigma\bet}_a \, \Omega^\alpha_\sigma \right]
\, M_{\alpha\bet} = 0.
\end{equation}
Equation \eqref{dphi} represents a system of purely algebraic
conditions on $M_{\alpha\bet}.$    We  then express the combined
algebraic systems \eqref{initalg} and \eqref{dphi} as
\begin{equation*}
\left(\Phi_1 \right)^{\alpha\bet}_a M_{\alpha\bet} = 0, \quad a=1,\dots, k_1.
\end{equation*}
We can proceed inductively to define a system of equations
\begin{equation}\label{willieirs}
\left(\Phi_{i+1}\right)^{\alpha\bet}_a M_{\alpha\bet} = 0, \quad a=1,\dots,
k_{i+1},
\end{equation}
where \eqref{willieirs} consists of the system $
(\Phi_i)^{\alpha\bet}_a M_{\alpha\bet} = 0$  along with the system
\begin{equation*}
[  \,(d \Phi_i)^{\alpha\bet}_a + 2 (\Phi_i)^{\alpha\sigma}_a
\Omega^\bet_\sigma ]  \, M_{\alpha\bet} = 0.
\end{equation*}
Clearly, at each point $z_0 = (x_0,y_0,u_0^\gamma) \in U \subset
{\mathbf R^2} \times {\mathbf R}^m$ and for all $j \geq 0 ,$ the
ranks of the matrices $\Phi_i$ satisfy
\begin{equation*}
0 \leq \text{Rank} \ \Phi_0 (z_0)  \leq \text{Rank}  \
\Phi_1 (z_0) \leq \dots \leq \text{Rank} \  \Phi_j (z_0)
\leq \dfrac{m(m+1)}{2}.
\end{equation*}
If given a point $z_0 \in U,$ we see that after  $ k \leq
m(m+1)/2$ differentiations, the rank  of the matrices $ \Phi_{i}
(z_0)$ must stabilize at some $0 \leq l \leq m(m+1)/2.$   More
precisely, at the point $z_0$ we have
\begin{equation*}
l=\text{Rank}  \ \Phi_k(z_0)   = \text{Rank}  \ \Phi_{k+i}(z_0),
\quad \forall \, i \geq 0.
\end{equation*}
Let $\{ C^1_{\alpha\gamma}, C^{1}_{\alpha\gamma}, \dots,
C^{s}_{\alpha\gamma}\},$ where $s=m(m+1)/2 - l,$ be a basis for
the set of solutions to the system of linear equations
\begin{equation}
\left(\Phi_{k}\right)^{\alpha\gamma}_a M_{\alpha\gamma} = 0
\end{equation}
at the point $z_0.$   In particular, $\{C^i_{\alpha\gamma}| \, i=1,\dots\}$ is
a linearly independent collection of constant bilinear forms such that
$(\Phi_{k})^{\alpha\gamma}_a (z_0)\,  C^i_{\alpha\gamma}=0.$

If the rank of $\Phi_{k} $ is constant in a neighborhood  $V \subset U$ of
$z_0,$ then a general result on systems of algebraic-differential equations,
which we prove in the appendix, states that there exists a neighborhood $W
\subset V$ of $z_0$ and a linearly independent collection of functions
$\{M^1_{\alpha\bet},M^2_{\alpha\bet},\dots,M^s_{\alpha\bet}\}$ such that
\begin{equation}\label{Epop1}
M^i_{\alpha\bet}(z_0) = C^{i}_{\alpha\bet},\quad (\Phi_{k})_a^{\alpha\bet}(z)
M^i_{\alpha\bet}(z) = 0, \quad  dM^i_{\alpha\bet} = M^i_{\alpha\sigma}
\Omega^\sigma_\bet + M^i_{\bet\sigma} \Omega^\sigma_\alpha,
\end{equation}
for all $z \in W,$ and $i=1,2,\dots,s.$  Moreover, we claim that
if a given collection of functions $M_{\alpha\bet}$ satisfies the
algebraic and differential conditions \eqref{Epop1}, then
 $M_{\alpha\bet}$ can be expressed as a linear combination
\begin{equation}\label{Ebob1}
M_{\alpha\bet} = \sum_{i=1}^s  c_i M^i_{\alpha\bet}, \quad  c_i \in \mathbf R.
\end{equation}
Indeed, if $M_{\alpha\bet}$ satisfies the algebraic condition
$(\Phi_{k})_a^{\alpha\bet}(z) M_{\alpha\bet}(z) = 0$ for all $z
\in W,$  it follows that $M_{\alpha\bet}$ can be uniquely
expressed as
\begin{equation}\label{Epop2}
M_{\alpha\bet}(z) = \sum c_i(z) \, M^i_{\alpha\bet}(z).
\end{equation}
If $M_{\alpha\bet}$ satisfies the differential condition
\eqref{dMomega}, then it follows from \eqref{Epop1} and
\eqref{Epop2} that $ M^i_{\alpha\bet} \, dc_i = 0. $ Since the
functions $M^i_{\alpha\bet}$ are pointwise linearly independent in
some neighborhood of $z_0,$ we deduce that $d c_i = 0,$ which in
turn implies that $c_i \in \mathbf R$ for $i=1,2,\dots,s.$   This
establishes \eqref{Ebob1} and completes the proof of the theorem.
\end{proof}

%
%

\section{Examples}

In this section we demonstrate our algorithm for solving the
variational multiplier problem on several examples. For each
example we calculate the invariants $H^\gamma_\alpha,$
$K^\gamma_\alpha,$ and $S^\gamma_{\alpha\bet}$ using \eqref{hks}
and we explicitly list the initial algebraic conditions
\eqref{MHMK} on a multiplier $M_{\alpha\bet}.$   We then determine
the differential condition \eqref{dMomega} and differentiate the
algebraic conditions to uncover any additional algebraic
constraints.   In each case we find the most general Lagrangian
and multiplier for the given system.

\begin{exam}  For our first example, consider the system
\begin{equation}\label{hopkins}
u_{xy} = v, \qquad v_{xy} = u.
\end{equation}
We will show that \eqref{hopkins} admits two Lagrangians. In this
case $S^\alpha_{\bet\gamma}=0,$  so that the only nontrivial
algebraic condition from \eqref{EMcon} is
\begin{equation}\label{doug}
M_{\alpha\gamma} H^\gamma_\bet = M_{\bet\gamma} K^\gamma_\alpha.
\end{equation}
The differential condition is $d M_{\alpha\bet} = 0.$  Using
\eqref{hks} we calculate
\begin{equation*}
{\mathbf H}= {\mathbf K} = \left[ \begin{matrix} 0 & 1 \\ 1 & 0
\end{matrix} \right],
\end{equation*}
and it follows from \eqref{doug} that the only algebraic
constraint is $ M_{11} = M_{22}.$   Since the differential
condition is  $dM_{\alpha\bet} =0,$ differentiating
$M_{11}=M_{22}$ produces no additional algebraic conditions. The
most general multiplier in this case is
\begin{equation*}
\mathbf M = \left[ \begin{matrix} a & b \\ b & a \end{matrix} \right], \quad
a,\, b \in \mathbf R, \ \ a^2 - b^2 \neq 0.
\end{equation*}
The Lagrangian corresponding to the multiplier $\mathbf M$ is
given by
\begin{equation*}
L = -\dfrac{a}{2}( u_x u_y + v_x v_y + 2uv ) -  \dfrac{b}{2}( 2
u_x v_y + u^2 + v^2 ).
\end{equation*}
A routine calculation shows that the Euler-Lagrange equations for
$L$ are
\begin{equation*}
E_1(L)= a(u_{xy} - v)+ b(v_{xy} -u), \qquad  E_2(L)= b(u_{xy} -
v)+a(v_{xy} - u).
\end{equation*}
\end{exam}

\begin{exam}  Consider the system
\begin{equation}\label{chuckd}
u_{xy} = v, \qquad v_{xy} = xu.
\end{equation}
The only difference between this system and the one in the first
example is the $x$ in the second equation.   We will show that
\eqref{chuckd} admits a unique Lagrangian.  The initial algebraic
conditions \eqref{EMcon} on the multiplier $M_{\alpha\bet}$ reduce
to
\begin{equation}\label{rundmc}
M_{11} = x M_{22}.
\end{equation}
The differential condition is again $ d M_{\alpha\bet} = 0.$
Differentiating \eqref{rundmc} and substituting $d M_{\alpha\bet}
= 0$ results in $M_{22} dx = 0,$ implying that $M_{22}=M_{11}=0.$
Further differentiations add no new algebraic conditions.
Consequently the multiplier ${\mathbf M},$ where $M_{11}=M_{22}=0$
and $M_{12}=M_{21}=1,$  is unique up to scalar multiplication. The
Lagrangian for \eqref{chuckd} is $L= -u_x v_y - (u^2 + x v^2)/2.$
\end{exam}

\begin{exam}  For our third example, we again we make
a slight change on the first example \eqref{hopkins} to get a
system that is not variational.   Let
\begin{equation}
u_{xy} = v, \qquad v_{xy} = u_x.
\end{equation}
In this case the algebraic and differential conditions
\eqref{EMcon} on the multiplier $M_{\alpha\bet}$ are given by
\begin{equation}\label{Emoney}
M_{11}  = 0, \quad dM_{11} = 2 M_{12} \, dy, \quad dM_{12} = M_{22} \, dy,
\quad dM_{22} = 0.
\end{equation}
Differentiating the algebraic condition $M_{11}=0$  we find that
$M_{12} \, dy = 0,$ implying that $M_{12}=0.$  Differentiating
$M_{12}=0,$ it follows that $M_{22} \, dy =0.$   Consequently, the
only solution to \eqref{Emoney} is the trivial solution ${\mathbf
M} \equiv 0.$
\end{exam}

\begin{exam} Consider the system
\begin{equation}\label{connectex}
u^\alpha_{xy} + \Gamma^{ \ \alpha}_{ \bet \ \gamma} (u^\epsilon) u^\bet_x
u^\gamma_y = 0,
\end{equation}
where $\Gamma^{ \ \alpha}_{\bet \ \gamma}$ are the components of a
symmetric connection on an $m$-dimensional manifold ${\mathcal
M}.$  We will show that \eqref{connectex} is variational if and
only if $\boldsymbol \Gamma$ is a metric connection.  Since
$\boldsymbol \Gamma$ is symmetric we have
$S^\gamma_{\alpha\beta}=0$ and the Laplace invariants for
\eqref{connectex} are given by $ H^\gamma_\alpha = R^{ \ \gamma}_{
\epsilon \ \alpha\sigma} u^\sigma_x u^\epsilon_y   \, $ and $\,
K^\gamma_\alpha = R^{ \ \gamma}_{ \sigma \ \alpha\epsilon}
u^\sigma_x u^\epsilon_y, \, $ where $ R^{ \ \alpha}_{ \epsilon \
\gamma\bet} $ are the components of the curvature tensor
associated to the connection $\boldsymbol \Gamma.$   Using
properties of the curvature tensor, we can show that the
conditions \eqref{EMcon}  on the multiplier are
\begin{equation} \label{curvaturecond}
M_{\alpha\gamma}  R^{ \ \gamma}_{\bet \ \epsilon\sigma} + M_{\bet\gamma} R^{\
\gamma}_{\alpha \ \epsilon\sigma} = 0, \qquad  dM_{\alpha\bet} = \left(
M_{\alpha\gamma} \Gamma^{ \ \gamma}_{\bet \ \epsilon} + M_{\bet\gamma}
\Gamma^{\ \gamma}_{\alpha \ \epsilon} \right) \, du^\epsilon.
\end{equation}
We see immediately from the differential condition in
\eqref{curvaturecond} that any multiplier $M_{\alpha\bet}$ must
satisfy ${\partial M_{\alpha\bet}} / {\partial x} = 0$ and
${\partial M_{\alpha\bet}}/{\partial y} = 0.$ Consequently,
$M_{\alpha\bet}=M_{\alpha\bet} (u^\epsilon).$ It follows that the
differential condition on the multiplier simplifies to
$\nabla_\gamma M_{\alpha\bet} = 0,$ where $\nabla_\gamma$ denotes
covariant differentiation with respect to $u^\gamma.$ This proves
that $\boldsymbol \Gamma$ is the Levi-Cevita connection for the
metric $M_{\alpha\bet}.$ Subsequent differentiations of the
algebraic condition \eqref{curvaturecond} imply that
$M_{\alpha\bet}$ must satisfy $M_{\alpha\gamma} \nabla_I \,
R^\gamma_{ \ \epsilon\bet\sigma} = M_{\bet\gamma} \nabla_I \,
R^\gamma_{ \ \sigma\alpha\epsilon},$ where $\nabla_I =
\nabla_{\alpha_1} \nabla_{\alpha_2} \dots \nabla_{\alpha_k}.$  We
remark that if $({\mathcal M}, \boldsymbol \Gamma)$ is a (locally)
symmetric space, then the algebraic condition
\eqref{curvaturecond} involving the curvature completely
determines the number of linearly independent metrics for the
given connection.
\end{exam}

%
%

\begin{exam}  Consider the system of differential
equations
\begin{equation}\label{LieAlgEx}
u^\alpha_{xy} + C^\alpha_{\bet\gamma} u^\bet_x u^\gamma_y = 0,
\end{equation}
where $C^\alpha_{\bet\gamma} \in \mathbf R$ are the structure
constants of an $m$ dimensional Lie algebra $\mathfrak g.$   The
constants $C^\alpha_{\bet\gamma}$ are skew symmetric in the lower
indices and satisfy the Jacobi identity.   Using \eqref{hks} we
calculate the invariants
\begin{gather*}
H^\gamma_\alpha = (C^\gamma_{\epsilon\sigma} C^\sigma_{\alpha\tau} -
C^\gamma_{\alpha\sigma} C^\sigma_{\epsilon\tau} ) u^\epsilon_x u^\tau_y \, ,
\qquad K^\gamma_\alpha = (C^\gamma_{\sigma\tau} C^\sigma_{\epsilon\alpha} -
C^\gamma_{\sigma\alpha} C^\sigma_{\epsilon\tau} ) u^\epsilon_x u^\tau_y \, ,
\\
S^\gamma_{\alpha\bet} = C^\gamma_{\alpha\bet} - C^\gamma_{\bet\alpha} = 2
C^\gamma_{\alpha\bet}.
\end{gather*}
Consequently, the algebraic conditions \eqref{EMcon} can be
summarized as
\begin{subequations}\label{algcondLie}
\begin{gather}\label{LieHK}
M_{\alpha\gamma} \left(C^\gamma_{\epsilon\sigma} C^\sigma_{\bet\tau} -
C^\gamma_{\bet\sigma} C^\sigma_{\epsilon\tau} \right)  =
M_{\bet\gamma}\left(C^\gamma_{\sigma\tau} C^\sigma_{\epsilon\alpha} -
C^\gamma_{\sigma\alpha} C^\sigma_{\epsilon\tau} \right),
\\ \label{torsoLie}
M_{\alpha\gamma} C^\gamma_{\bet\epsilon} + M_{\bet\gamma}
C^\gamma_{\alpha\epsilon} = 0.
\end{gather}
\end{subequations}
The differential condition \eqref{dMomega} on $M_{\alpha\bet}$ is
$ dM_{\alpha\bet} = (M_{\alpha\sigma} C^\sigma_{\bet\gamma} +
 M_{\bet\sigma} C^\sigma_{\alpha\gamma})du^\gamma,
$ so that $dM_{\alpha\bet}=0$ as a result of \eqref{torsoLie}. It is easy to
check that \eqref{torsoLie} implies \eqref{LieHK}. It follows that there is a
Lagrangian for \eqref{LieAlgEx} with multiplier $M_{\alpha\bet}$ if and only if
$M_{\alpha\bet}$ is constant and $M_{\alpha\bet}$ satisfies equation
\eqref{torsoLie}.  Moreover, the Lagrangian is given by
\begin{equation*}
L= -\frac{1}{6} \, M_{\alpha\bet}  \, \left(3 u^\alpha_x u^\bet_y   -  2
C^\alpha_{\epsilon\tau}  u^\bet u^\epsilon_x u^\tau_y \right).
\end{equation*}
We remark that \eqref{torsoLie} is exactly the same as the condition for the
existence of a bi-invariant symmetric bilinear form for a Lie algebra
$\mathfrak g$ with structure constants $C^\alpha_{\bet\gamma}.$ If $\mathfrak
g$ is semi-simple, the Killing form provides us with a non-degenerate solution
to \eqref{torsoLie}. Consequently, \eqref{LieAlgEx} is variational whenever
$\mathfrak g$ is semi-simple. Moreover, the number of solutions to
\eqref{torsoLie} is equal to the dimension of the Lie algebra cohomology space
$H_3(\mathfrak g).$ If $\mathfrak g$ is simple, then $\dim H_3(\mathfrak g)=1$
and the Killing form determines the only non-degenerate solution to
\eqref{torsoLie} up to a scalar multiple (See \cite{Koszul1950}, Theorems 11.1,
11.2). We remark that semi-simplicity is not a necessary condition for
\eqref{torsoLie} to hold, as there are solvable Lie algebras which also admit
bi-invariant bilinear forms.  For example, consider the solvable
$4$-dimensional Lie algebra $\mathfrak g$ of consisting of real matrices of the
form
\begin{equation*}
{\mathbf A} = \left[ \begin{matrix} 0 & a_1 & a_2 \\ 0 & a_3 & a_4 \\ 0 & 0 & 0
\end{matrix} \right].
\end{equation*}
It is easy to check that the bilinear map ${\mathbf M}:{\mathfrak g} \times
{\mathfrak g} \to {\mathbf R}$ defined by
\begin{equation*}
M(\mathbf A,\mathbf B) = \lambda a_3 b_3 + \mu ( a_2 b_3  + a_3 b_2 - a_1 b_4 -
a_4 b_1)
\end{equation*}
is  non-degenerate and bi-invariant for all $\mu \neq 0$ and all $\lambda \in
{\mathbf R}.$
\end{exam}

%
%
%
%
%
%

\section{Classification of Variational Systems in Two Dependent
Variables}

In this section we establish a result characterizing the variational systems
\begin{equation}\label{11gordies}
u_{xy}^\alpha = f^\alpha(x,y,u^\gamma,u^\gamma_x,u^\gamma_y),
\end{equation}
of two equations and two dependent variables that admit multiple
Lagrangians. In order to proceed, we need to make precise two
concepts that are paramount to our discussion.  We first define
what it means for a system to have multiple Lagrangians.  Then we
review the notion of contact equivalence of two systems of
differential equations \eqref{11gordies}.

We say that a system of differential equations \eqref{11gordies}
\emph{admits $k$ Lagrangians} if there exists a set of linearly
independent Lagrangians $\{L_1,$ $L_2, \dots, L_k \}$ and a set of
linearly independent variational multipliers $\{ M^1_{\alpha\bet},
\dots, M^k_{\alpha\bet}\},$ such that $E_\alpha(L_i) =
M^i_{\alpha\bet} (u^\bet_{xy} - f^\bet)$ for $i=1,\dots,k.$  We
say  that two $f$-Gordon systems \eqref{11gordies} are
\emph{contact equivalent} if there exists a local diffeomorphism
\begin{equation*}
\Phi:J^2({\mathbf R}^2, {\mathbf R}^m) \to J^2({\mathbf R}^2,
{\mathbf R}^m),
\end{equation*}
where $J^2({\mathbf R}^2,{\mathbf
R}^m)$ denotes the second-order jet-bundle of local sections
$s:{\mathbf R}^2 \to {\mathbf R}^m,$ such that
\begin{equation*}
\Phi^* (\bar u^\alpha_{\bar x \bar y} - \bar{f}^\alpha) = Q^\alpha_\gamma
(u^\gamma_{xy} - f^\gamma),
\end{equation*}
and $\Phi^* \bar {\mathcal C} \subset {\mathcal C},$ where ${\mathcal C}$ is
the ideal generated by the 1-forms
\begin{equation*}
du^\alpha - u_x^\alpha \, dx - u_y^\alpha \, dy, \quad du^\alpha_x
- u_{xx}^\alpha \, dx - u_{xy}^\alpha \, dy, \quad du^\alpha_y -
u^\alpha_{xy} \, dx - u^\alpha_{yy}\, dy.
\end{equation*}
It was shown in \cite{mydiss} that any contact equivalence $\Phi$
of two systems of the form \eqref{11gordies} is the prolongation
of a fiber preserving transformation
\begin{equation}\label{transtype}
\bar x = A(x), \quad \bar y = B(y), \quad \bar u^\alpha =
C^\alpha(x,y,u^\gamma),
\end{equation}
up to an interchange $x \leftrightarrow  y$ of the independent
variables.

We have the following theorem that completely characterizes the
variational $f$-Gordon systems \eqref{11gordies} in two dependent
variables which admit two or more inequivalent Lagrangians.
\begin{theorem}\label{TheoClassify}  Let $\calr$ denote the system
\begin{equation} \label{2by2system}
u_{xy} = f(x,y,u,v,u_x,v_x,u_y,v_y), \qquad v_{xy} =
g(x,y,u,v,u_x,v_x,u_y,v_y).
\end{equation}

\smallskip

\noindent 1) $\calr$ admits three Lagrangians if and only if $\calr$ is contact
equivalent to a system
\begin{equation*}
u_{xy} = \lambda(x,y) u,  \qquad  v_{xy} = \lambda(x,y) v.
\end{equation*}

\smallskip

\noindent 2)  $\calr$ admits two Lagrangians if and only if $\calr$ is contact
equivalent to a system
\begin{equation*}
u_{xy} = W_v(x,y,u,v), \qquad v_{xy} = W_u(x,y,u,v),
\end{equation*}
where $W$ satisfies one of $W_{uu} + W_{vv}=0,$ $W_{uu}=W_{vv},$ or $W_{vv} =
0.$
\end{theorem}

\begin{remark}  According to Proposition \eqref{PropMcond} and Remark
\eqref{RMcon}, a system \eqref{2by2system} admits at most three
Lagrangians.   Moreover, since the dimension of the vector space
of symmetric $m \times m$ matrices is $m(m+1)/2,$ a system
\eqref{11gordies} of $m$ equations admits at most $m(m+1)/2$
Lagrangians.  If a system of $m$ equations \eqref{11gordies}
admits the maximal number of Lagrangians, then it can be shown the
given system is contact equivalent to a system $u^\alpha_{xy} =
\lambda(x,y) u^\alpha.$
\end{remark}

The proof of the Theorem \eqref{TheoClassify} depends on the
following lemma
 that completely characterizes the $f$-Gordon
equations $u^\alpha_{xy} = g^\alpha(x,y,u^\gamma)$ up to contact
equivalence. The proof of Lemma \eqref{PropHeqK} is quite tedious
and is delayed until after the proof of Theorem
\eqref{TheoClassify}.

\begin{lemma}\label{PropHeqK}   A system of partial differential equations
$u^\alpha_{xy}=f^\alpha(x,y,u^\gamma,u^\gamma_x,u^\gamma_y)$  is
contact equivalent to a system
$u^\alpha_{xy}=g^\alpha(x,y,u^\gamma)$ if and only if ${\mathbf H}
= {\mathbf K}$ and $S^\alpha_{\bet\gamma} = 0.$   If the number of
dependent variables $m>1$ and  ${\mathbf H}={\mathbf K}=\lambda
{\mathbf I},$ then $\lambda=\lambda(x,y)$ and the given system is
contact equivalent to a system $u^\alpha_{xy} = \lambda(x,y)
u^\alpha.$
\end{lemma}

\noindent{\it{Proof of Theorem 4.1.}} We first show that if a
system \eqref{2by2system} admits multiple Lagrangians, then
\eqref{2by2system} satisfies the hypotheses of Lemma
\eqref{PropHeqK} and is contact equivalent to a system of the form
\begin{equation}\label{simplesystem}
u_{xy} = F(x,y,u,v), \qquad v_{xy} = G(x,y,u,v).
\end{equation}
According to Proposition \eqref{PropMcond}, if \eqref{2by2system}
admits a first-order Lagrangian $L$ with a symmetric multiplier
$M_{\alpha\bet},$ then $M_{\alpha\bet}$ satisfies the algebraic
conditions
\begin{equation}\label{dookie}
M_{\alpha\gamma} H^\gamma_\bet=M_{\bet\gamma} K^\gamma_\alpha,
\qquad M_{\alpha\sigma} S^\sigma_{\bet\gamma}+ M_{\bet\sigma}
S^\sigma_{\alpha\gamma} = 0,
\end{equation}
where $H^\alpha_\gamma, \ K^\gamma_\alpha,$ and
$S^\alpha_{\bet\gamma}$ are given by \eqref{hks}.  Multiplying the
second equation of \eqref{dookie} by $M^{\alpha\bet}$ yields
\begin{equation}\label{Sbye}
0=M^{\alpha\bet} M_{\alpha\sigma} S^\sigma_{\bet\gamma}+ M^{\alpha\bet}
M_{\bet\sigma} S^\sigma_{\alpha\gamma} = \delta^\bet_\sigma
S^\sigma_{\bet\gamma} + \delta_\alpha^\sigma S^\sigma_{\alpha\gamma} = 2
S^\alpha_{\alpha\gamma}.
\end{equation}
If their are only two dependent variables, then equation
\eqref{Sbye} implies that $S^1_{12}=S^2_{12}= 0.$ Since
$S^\alpha_{\bet\gamma}=-S^\alpha_{\gamma\bet},$ we have
$S^\alpha_{\bet\gamma}=0$ for all $\alpha,\beta,\gamma=1,2.$  Then
the algebraic conditions \eqref{dookie} can be expressed as
\begin{equation}\label{systemformults}
{\mathbf A} \cdot \left[
\begin{matrix}  M_{11} \\ M_{12} \\ M_{22}  \end{matrix} \right] =
{\mathbf 0}, \qquad   {\mathbf A} =
\left[ \begin{matrix}  H^1_1 - K_1^1   &  H^2_1 - K^2_1  &  0 \\
0  &  H^1_2 - K_2^1  & H^2_2 - K^2_2 \\  H^1_2 & H^2_2 - K^1_1 & -K^2_1  \\
K^1_2  &  K^2_2 -  H^1_1  &  -H_1^2 \end{matrix} \right]  .
\end{equation}
If \eqref{11gordies} admits 2 or more inequivalent Lagrangians,
then there are least two linearly  independent solutions to
\eqref{systemformults} and it follows that the rank of ${\mathbf
A}$ is at most one.   Moreover, the rank of $\mathbf A$ is one or
less if and only if ${\mathbf H} = {\mathbf K}$ and is rank zero
if and only if ${\mathbf H}={\mathbf K} = \lambda I.$   From Lemma
\eqref{PropHeqK}, we deduce that $\text{Rank} \ {\mathbf A} \leq
1$ if and only if \eqref{2by2system} is contact equivalent to
\eqref{simplesystem}.

We will complete the proof of the theorem by analyzing the
algebraic conditions \eqref{systemformults} for systems of the
form \eqref{simplesystem}.  A calculation of ${\mathbf H}$ and
${\mathbf K}$ for \eqref{simplesystem} reveals that
\begin{equation*}
{{\mathbf H}} = {{\mathbf K}} = \left[ \begin{matrix}  F_u &  F_v \\
G_u & G_v
\end{matrix} \right].
\end{equation*}
In this case the algebraic conditions \eqref{systemformults} for
the existence of a multiplier $M_{\alpha\bet}$  simplify to
\begin{equation}\label{simpalg}
F_v M_{11}  + (G_v - F_u) M_{12}   -  G_u  M_{22} = 0
\end{equation}
The differential condition \eqref{dMomega} reduces to
$dM_{\alpha\bet} = 0.$ Consequently, solving the variational
multiplier problem for \eqref{simplesystem} is equivalent to
determining all constant solutions $M_{\alpha\bet}$ to the
equation \eqref{simpalg}.

There are 3 linearly independent solutions to \eqref{simpalg} if
and only if ${\mathbf H}={\mathbf K} = \lambda {\mathbf I}.$ In
this case \eqref{simplesystem} is contact equivalent to a system
\begin{equation}\label{3lagcyst}
u_{xy} = \lambda(x,y) \, u, \qquad v_{xy} = \lambda(x,y) \, v.
\end{equation}
The most general Lagrangian for \eqref{3lagcyst} is  a linear
combination of the Lagrangians $L_1 =  u_{x} u_y  +  \lambda u^2,
\, $ $L_2 = u_x v_y + \lambda uv, \, $ and $L_3 = v_x v_y +
\lambda v^2.$

We  analyze the case where the rank of $\mathbf A$ is exactly one
and we assume there are two non-degenerate, linearly independent,
constant solutions $M^1_{\alpha\bet}$ and $M^2_{\alpha\bet}$ to
\eqref{simpalg}.   We claim there exists an indefinite multiplier
${\mathbf M}= (M_{\alpha\bet})$ that satisfies \eqref{simpalg}. If
one of $\det {\mathbf M}_1<0$ or $\det {\mathbf M}_2<0,$ then we
are done, so we assume that $\det {\mathbf M}_1  > 0$ and $\det
{\mathbf M}_2 > 0.$   If
\begin{equation}\label{m1m2}
{\mathbf M}_1 = \left[ \begin{matrix} a & b \\ b & c \end{matrix}\right],
\qquad  {\mathbf M}_2 = \left[ \begin{matrix} p & q \\ q & r
\end{matrix}\right],
\end{equation}
then it follows that $ac > 0$ and $pr >  0.$   We claim there is a
scalar $\mu$ such that $\det ( {\mathbf M_1} - \mu {\mathbf M_2} )
< 0.$ Indeed, if
\begin{equation}\label{polym}
 P(\mu) = \det ( {\mathbf M}_1 - \mu {\mathbf M}_2 ) = (\det
{\mathbf M}_2)^2 \, \mu^2  +  (2bq -ar- pc ) \, \mu   +  (\det {\mathbf
M}_1)^2,
\end{equation}
then the discriminant $\Delta$ can be expressed as
\begin{equation}\label{discrim}
\Delta = \dfrac{(aq-bp)^4 + 2 (p^2 \det {\mathbf M}_1 + a^2 \det {\mathbf
M}_2)(aq - bp)^2   + (a^2 {\mathbf M}_2 - p^2 {\mathbf M}_1)^2 }{a^2 p^2}.
\end{equation}
It follows from \eqref{m1m2} and \eqref{discrim} that $\Delta \geq
0$ with equality holding if and only if ${\mathbf M}_1 = t \,
{\mathbf M}_2$ for some $t \in {\mathbf R}.$  Consequently, the
polynomial  \eqref{polym} has two real roots and there exists
$\mu \in {\mathbf R}$ such that $\det ( {\mathbf M}_1 - \mu
{\mathbf M}_2 ) < 0.$  Now we have established that if there are
two independent solutions to \eqref{simpalg}, with at least one of
the solutions non-degenerate, then there is an indefinite
multiplier $M_{\alpha\bet}.$

If we make a linear change of variables $u^\alpha =
T^\alpha_\gamma \bar u^\gamma,$  then a direct calculation of the
Euler-Lagrange equations for the Lagrangian $L$ and the
transformed Lagrangian $\bar L$ verifies that the corresponding
variational multipliers transform according to the rule
\begin{equation}\label{Mtransrule}
M_{\alpha\beta} = T^\sigma_\alpha T^\tau_\bet \bar M_{\sigma\tau}.
\end{equation}
After a linear change of variables $u^\alpha \to T^\alpha_\gamma
u^\gamma,$  we may then assume that the indefinite multiplier
$M_{\alpha\bet}$ is of the form
\begin{equation}\label{thisM}
\mathbf M= \left( \begin{matrix}  0  & 1 \\ 1  &  0
\end{matrix} \right).
\end{equation}
Substituting \eqref{thisM} into \eqref{simpalg} implies that $G_v
- F_u= 0.$  As a consequence of the de Rham theorem, we see there
exists a smooth function $W(x,y,u,v)$ such that
\begin{equation}\label{WvF}
W_v = F, \qquad   W_u = G.
\end{equation}
There is a second multiplier ${\mathbf N},$ independent of
${\mathbf M},$ satisfying \eqref{simpalg}.  We may assume,
possibly after subtracting a scalar multiple of $\mathbf M,$ that
\begin{equation*}
\mathbf N = \left( \begin{matrix}  a  & 0 \\ 0  &  b
\end{matrix} \right).
\end{equation*}
According to \eqref{Mtransrule}, a linear transformation $u \to
\lambda u,$ $v \to
 \lambda^{-1} v$ preserves $\mathbf M$ and transforms
${\mathbf N}$ as
\begin{equation*}
\mathbf N \to  \left( \begin{matrix}  a \lambda^2  & 0 \\ 0  &
 b\lambda^{-2}
\end{matrix} \right).
\end{equation*}
We may then assume,  possibly after a scaling ${\mathbf N} \to k {\mathbf N},$
that ${\mathbf N}$ has the form
\begin{equation*}
\mathbf N= \left( \begin{matrix}  1  & 0 \\ 0  &  \varepsilon
\end{matrix} \right),
\end{equation*}
where $\varepsilon=0,1,$ or $-1.$   Taking into account that
\eqref{WvF} holds, substituting ${\mathbf N}$  into
\eqref{simpalg} implies that
\begin{equation*}
W_{vv}  - \varepsilon W_{uu}  = 0.
\end{equation*}
If $\varepsilon=-1,$ then $W$ satisfies Laplace's equation $W_{uu} + W_{vv}=0$
and there exists a function $Z(x,y,u,v)$ such that $Z_u=W_v = F$ and $-Z_v =
W_u = G.$ The most general Lagrangian in this case is given by
\begin{equation*}
L=c_1 ( u_x  v_y  + W)  + c_2 ( u_x u_y - v_x v_y  + Z).
\end{equation*}
If $\varepsilon=0,$  then $W_{vv}=0$ and \eqref{simplesystem} can
be expressed
\begin{equation}\label{vvvroom}
u_{xy} = a'(u), \qquad  v_{xy} = a''(u) v + b'(u).
\end{equation}
The most general Lagrangian for \eqref{vvvroom} is
\begin{equation*}
L= c_1 [u_{x} u_y + 2 a(u)] + c_2 [ u_x v_y  + a'(u) v + b(u) ].
\end{equation*}
If $\varepsilon=1,$ then $W$ satisfies the wave equation and
$W=W_1(u+v) + W_2(u-v).$ The most general Lagrangian in the case
is
\begin{equation*}
L = c_1 [ u_x u_y  +  v_x v_y  +   2 W_1(u+v) - 2 W_2(u-v) ]
  + c_2 [ u_x v_y  +  W_1(u+v) +  W_2(u-v) ].
\end{equation*}
This establishes the second statement of the theorem. \qed

\medskip

\noindent{\it{Proof of Lemma 4.3.}}  If we assume that two systems
${\bar u}_{\bar x \bar y} = {\bar f}^\alpha$ and $u_{xy} =
f^\alpha$ are contact equivalent with the change of coordinates
given by \eqref{transtype}, then the formulas \eqref{hks} and a
tedious application of the chain rule  will verify that the
transformation rules for ${\mathbf H},$ ${\mathbf K},$ and
${\mathbf S}$ are
\begin{equation}\label{transtheo}
\bar H^\alpha_\bet = \dfrac{1}{A'B'} \dfrac{\partial \bar u^\alpha}{\partial
u^\sigma}  \dfrac{\partial u^\tau}{\partial \bar u^\bet} \,  H^\sigma_\tau, \ \
\ \bar K^\alpha_\bet = \dfrac{1}{A'B'} \dfrac{\partial \bar u^\alpha}{\partial
u^\sigma}  \dfrac{\partial u^\tau}{\partial \bar u^\bet}  \, K^\sigma_\tau, \ \
\ \bar S^\alpha_{\bet\gamma} = \dfrac{\partial \bar u^\alpha}{\partial
u^\sigma} \dfrac{\partial u^\tau}{\partial \bar u^\beta} \dfrac{\partial
u^\epsilon}{\partial \bar u^\gamma} \, S^\sigma_{\tau\epsilon},
\end{equation}
where $({\partial  u^\alpha}/{\partial \bar u^\gamma})  \cdot (
{\partial \bar u^\gamma}/{\partial u^\bet}) = \delta^\alpha_\bet.$
It follows from \eqref{transtheo} that the conditions ${\mathbf H}
={\mathbf K}$ and ${\mathbf S}=0$ are invariant with respect to
transformation \eqref{transtype}.

For a system $u^\alpha_{xy} = g^\alpha(x,y,u^\gamma),$ using
\eqref{hks} we see that
\begin{equation}\label{hksforgu}
H^\alpha_\gamma = K^\alpha_\gamma = \dfrac{\partial g^\alpha}{\partial
u^\gamma}, \qquad \quad S^\alpha_{\bet\gamma} = 0.
\end{equation}
From \eqref{transtheo} and \eqref{hksforgu}, we deduce that any
system $u^\alpha_{xy} =
f^\alpha(x,y,u^\gamma,u^\gamma_x,u^\gamma_y)$ that is contact
equivalent to $u_{xy}^\alpha = g^\alpha(x,y,u^\gamma)$ must
necessarily have ${\mathbf H} = {\mathbf K}$ and
$S^\alpha_{\bet\gamma} = 0.$

We now prove that any system
\begin{equation}\label{fgord4}
u^\alpha_{xy} = f^\alpha(x,y,u^\gamma,u^\gamma_x,u^\gamma_y)
\end{equation}
with the property that ${\mathbf H}={\mathbf K}$ and
$S^\alpha_{\bet\gamma}=0$ is equivalent to a system $u_{xy}^\alpha
= g^\alpha(x,y,u^\gamma).$   It follows from \eqref{hks} that if
$H^\alpha_\gamma = K^\alpha_\gamma,$ then
\begin{equation*}
\dfrac{\partial^2 f^\alpha}{\partial u^\gamma_x \partial u^\bet_x} = 0, \quad
\dfrac{\partial^2 f^\alpha}{\partial u^\gamma_y  \partial u^\bet_y} = 0, \quad
\forall \, \alpha,\beta, \gamma = 1,2,\dots,m.
\end{equation*}
Consequently,  equation \eqref{fgord4} simplifies to an equation
of the form
\begin{equation}\label{nicef}
u^\alpha_{xy} + C^\alpha_{\bet\gamma} u^\bet_x u^\gamma_y + A^\alpha_\gamma
u^\gamma_x  + B^\alpha_\gamma u_y^\gamma + G^\alpha=0,
\end{equation}
where $C^\alpha_{\bet\gamma}, \, A^\alpha_\gamma, \,
 B^\alpha_\gamma,$ and $G^\alpha$ are functions of $x,y,$ and
$u^\epsilon.$   Moreover, $S^\alpha_{\bet\gamma} = 0$ if and only
if $C^\alpha_{\bet\gamma} = C^\alpha_{\gamma\bet}.$  With a
judicious choice of coordinates,  we will now eliminate the
quadratic terms of \eqref{nicef}. If we let $u^\alpha =
g^\alpha(x,y,\bar u^\gamma),$ then \eqref{nicef} transforms as
\begin{equation}
\dfrac{\partial g^\alpha}{\partial \bar u^\bet} \bar u^\bet_{xy} +
\left(\dfrac{\partial g^\alpha}{\partial \bar u^\bet
\partial \bar u^\gamma} - C^\alpha_{\epsilon\tau} (x,y,g^\tau)  \dfrac{\partial
g^\epsilon}{\partial \bar u^\bet}  \dfrac{\partial g^\tau}{\partial \bar
u^\gamma}  \right) \bar  u^\bet_x \bar u^\gamma_y + \bar A^\alpha_\gamma \bar
u^\gamma_x  +  \bar B^\alpha_\gamma \bar u^\gamma_x + \bar G^\alpha.
\end{equation}
We see that $\bar C^\alpha_{\bet\gamma} = 0$ whenever $g^\alpha$ satisfies
\begin{equation}\label{Ctrans}
\dfrac{\partial^2 g^\alpha}{\partial \bar u^\bet
\partial \bar u^\gamma} +  C^\alpha_{\epsilon \tau} (x,y,g^\tau)  \dfrac{\partial
g^\epsilon}{\partial \bar u^\bet}  \dfrac{\partial g^\tau}{\partial \bar
u^\gamma} = 0
\end{equation}
We differentiate \eqref{Ctrans} with respect to $\bar u^\delta,$
and after substituting from \eqref{Ctrans} and skew-symmetrizing
over $\beta$ and $\delta,$ we obtain the integrability conditions
on \eqref{Ctrans}
\begin{equation}\label{Cint}
\left( \dfrac{\partial C^\alpha_{\epsilon\sigma}}{\partial u^\tau} -
\dfrac{\partial C^\alpha_{\tau\sigma}}{\partial u^\epsilon} +
 C^\alpha_{\tau\mu} C^\mu_{\sigma\epsilon}  - C^\alpha_{\epsilon\mu}
C^\mu_{\sigma\tau} \right) \dfrac{\partial g^\epsilon}{\partial \bar u^\bet}
\dfrac{\partial g^\sigma}{\partial \bar u^\gamma} \dfrac{\partial
g^\tau}{\partial \bar u^\delta}  = 0.
\end{equation}
On the other hand a calculation of ${\mathbf H}$ and ${\mathbf K}$
for \eqref{nicef} yields
\begin{equation}\label{hkquad}
\dfrac{\partial^2}{\partial u^\epsilon_x \partial u^\tau_y}
\left(H^\alpha_\sigma - K^\alpha_\sigma\right) = \dfrac{\partial
C^\alpha_{\epsilon\sigma}}{\partial u^\tau} - \dfrac{\partial
C^\alpha_{\tau\sigma}}{\partial u^\epsilon} +
 C^\alpha_{\tau\mu} C^\mu_{\sigma\epsilon}  - C^\alpha_{\epsilon\mu}
C^\mu_{\sigma\tau}.
\end{equation}
As a consequence of \eqref{hkquad}, the integrability conditions
\eqref{Cint} are satisfied whenever ${\mathbf H}={\mathbf K}.$ The
system of partial differential equations \eqref{Ctrans} then
satisfies the Frobenius condition and we deduce that there exists,
at least locally, a non-degenerate collection of functions
$g^\alpha$ satisfying \eqref{Ctrans}.

We may now  assume that
$u^\alpha_{xy}=f^\alpha(x,y,u^\gamma,u^\gamma_x,u^\gamma_y)$ is of the form
\begin{equation}\label{fnoC}
u^\alpha_{xy} + A^\alpha_\gamma(x,y,u^\epsilon) u^\gamma_x  +
B^\alpha_\gamma(x,y,u^\epsilon) u^\gamma_y + G^\alpha(x,y,u^\epsilon)=0.
\end{equation}
For \eqref{fnoC}, we calculate
\begin{equation}\label{hminusk}
H^\alpha_\gamma - K^\alpha_\gamma = \frac{\partial A^\alpha_\gamma}{\partial
u^\bet}  u^\bet_x + \dfrac{\partial B^\alpha_\bet}{\partial u^\gamma} u^\bet_y
+  A^\sigma_\gamma B^\alpha_\sigma - B^\sigma_\gamma A^\alpha_\sigma +
\dfrac{\partial A^\alpha_\gamma}{\partial x}  - \dfrac{\partial
B^\alpha_\gamma}{\partial y}.
\end{equation}
Evidently, $H^\alpha_\gamma = K^\alpha_\gamma$ only if ${\partial
B^\alpha_\gamma}/{\partial u^\bet} = {\partial
A^\alpha_\gamma}/{\partial u^\bet} = 0$ for all $\alpha,\beta,$
and $\gamma.$   It follows that \eqref{fnoC} simplifies to
\begin{equation}\label{almostthere}
u_{xy}^\alpha +  A^\alpha_\gamma(x,y) u^\gamma_x  + B^\alpha_\gamma(x,y)
u^\gamma_y  + G^\alpha(x,y,u^\bet) = 0.
\end{equation}
We now eliminate the functions $A^\alpha_\gamma(x,y)$ and
$B^\alpha_\gamma(x,y)$ from \eqref{almostthere} with a
transformation $u^\alpha \to N^\alpha_\gamma(x,y) u^\gamma,$ where
$N^\alpha_\gamma$ satisfies the system of partial differential
equations
\begin{equation}\label{Ncond}
\dfrac{\partial N^\alpha_\gamma}{\partial y \ }  + A^\alpha_\tau  \,
N^\tau_\gamma = 0, \qquad \dfrac{\partial N^\alpha_\gamma}{\partial x  \ }  +
B^\alpha_\tau \, N^\tau_\gamma = 0.
\end{equation}
The integrability conditions for \eqref{Ncond} are given by
\begin{equation}\label{Nint}
\dfrac{\partial A^\alpha_\gamma}{\partial x \ } + A^\tau_\gamma B_\tau^\alpha =
\dfrac{\partial B^\alpha_\gamma}{\partial y \ } + B^\tau_\gamma A_\tau^\alpha.
\end{equation}
As a consequence of \eqref{hminusk} we see that \eqref{Nint} holds
whenever ${\mathbf H}={\mathbf K}.$   It follows that the system
\eqref{Ncond} satisfies the Frobenius condition and there exists
functions $N^\alpha_\gamma(x,y)$ such that $\det N^\alpha_\gamma
\neq 0$ and \eqref{Ncond} is satisfied. Moreover, we arrive at an
equation of the desired form
\begin{equation}\label{guonly}
u^\alpha_{xy} = g^\alpha(x,y,u^\gamma).
\end{equation}

To prove the final statement of the lemma, we assume that
\eqref{guonly} has the property that
$H^\alpha_\gamma=K^\alpha_\gamma = \lambda \delta^\alpha_\gamma.$
A routine calculation shows that $H^\alpha_\gamma =
K^\alpha_\gamma={\partial g^\alpha}/{\partial u^\gamma}.$  It is
evident that $H^\alpha_\gamma = {\partial g^\alpha}/{\partial
u^\gamma } = \lambda \delta^\alpha_\gamma$ if and only if
${\partial g^\alpha}/{\partial u^\bet} = 0$ whenever $\beta \neq
\alpha.$ This means that each function $g^\alpha$ satisfies
\begin{equation}\label{gdiag}
\dfrac{\partial g^\alpha}{\partial u^\alpha} = \lambda ,  \quad
\alpha=1,2,\dots,m,
\end{equation}
with no summation on $\alpha.$  Differentiating equation
\eqref{gdiag} with respect to $u^\bet$, $ 1 \leq \beta \neq \alpha
\leq m$, results in ${\partial \lambda}/{\partial u^\bet} = 0$ for
all $\beta.$ We deduce that $\lambda=\lambda(x,y)$ and it becomes
apparent from \eqref{gdiag} that $g^\alpha = \lambda(x,y) u^\alpha
+ k^\alpha(x,y).$ After a transformation $u^\alpha \to u^\alpha +
m^\alpha(x,y)$, where $m^\alpha_{xy} = \lambda m^\alpha +
k^\alpha,$  we see that \eqref{guonly} is equivalent to
$u^\alpha_{xy} = \lambda(x,y) u^\alpha.$ \qed

%
%

\section{The Variational Bicomplex For Systems of PDE}
In this section we introduce some basic definitions and results
used in our solution to the variational multiplier problem in
Section 6, including infinite jet bundles and variational
bicomplexes. As we are only interested in applications to our
study of the variational multiplier problem, our discussion will
be of a rather brief nature.  For a  detailed and intrinsic
construction of variational bicomplex, we refer the reader to
\cite{And1992}, \cite{AndKam1997}, \cite{Kamran2002}, and
\cite{Vino1984}.

Let $\pi^k:J^k(E) \to {\mathbf R}^n$ denote the bundle of $k$-jets of local
sections of the trivial bundle $E={\mathbf R}^n \times {\mathbf R}^m.$ Local
coordinates for $J^k(R^n,R^m)$ are given by
\begin{equation*}
(x^i,u^\alpha,u^\alpha_i, u^\alpha_{i_1 i_2}, u^\alpha_{i_1 i_2
i_3}, \dots, u^\alpha_{i_1 i_2 \dots i_k} ),
\end{equation*}
where $1 \leq i_1 \leq i_2 \leq \dots i_k \leq n$ and $1 \leq
\alpha \leq m.$ There are natural projections $\pi^l_k: J^l(E) \to
J^k(E)$ for $ l \geq k.$ The \emph{infinite jet bundle} over E,
$J^\infty(E)$, is defined as the inverse limit of the sequence of
finite jet bundles $\{ J^k(E) \, | \, k=0,1,2, \dots\},$ along
with the projections $\pi^\infty_k:J^\infty(E) \to J^k(E)$ and
$\pi^\infty:J^\infty(E) \to {\mathbf R}^n.$  The \emph{contact
ideal}  ${\mathcal C}(J^\infty(E))$   is generated by the 1-forms
\begin{equation*}
\theta^\alpha_{i_1 i_2 \dots i_k}  = du^\alpha_{i_1 i_2 \dots i_k}
- u^\alpha_{i_1 i_2 \dots i_k j} dx^j, \qquad   \forall \,
k=0,1,2,\dots \ \,  .
\end{equation*}
The full exterior algebra $\Omega^*(J^\infty(E))$ of differential
forms on $J^\infty(E)$ is generated by the 1-forms
\begin{equation*}
dx^i, \, \theta^\alpha, \, \theta^\alpha_{i}, \, \theta^\alpha_{i
j},  \dots \ \, .
\end{equation*}
There is a  bi-grading of the differential forms on $J^\infty(E),$
\begin{equation*}
\Omega^p(J^\infty(E)) = \bigoplus_{r+s=p} \Omega^{r,s}
(J^\infty(E)),
\end{equation*}
where $\Omega^{r,s} ( J^\infty(E))$ is the
$C^\infty(J^\infty(E))$-module generated by differential forms of
the type
\begin{equation*}
 dx^{i_1} \wedge dx^{i_2} \wedge \dots \wedge dx^{i_r}\wedge
\theta^{\alpha_1}_{J_1} \wedge \dots \wedge
\theta^{\alpha_s}_{J_s}.
\end{equation*}

The exterior derivative $d: \Omega^{p}(J^\infty(E)) \rightarrow
\Omega^{p+1}(J^\infty(E))$ splits into the horizontal and vertical
differentials $d=d_H + d_V,$ where
\begin{equation*}
d_H: \Omega^{r,s}(J^\infty(E)) \rightarrow \Omega^{r+1,s}(J^\infty(E)), \qquad
d_V:\Omega^{r,s}(J^\infty(E)) \rightarrow \Omega^{r,s+1}(J^\infty(E)).
\end{equation*}
Since $d^2 = 0,$ it follows that $d_H^2 = 0, \, $ $d_V^2 = 0, \, $
and $ d_H d_V = - d_V d_H.$ The local coordinate expressions for
the horizontal and vertical derivatives of a smooth functions $f
\in C^\infty(J^\infty(E))$ and $1$-forms $dx^i$ and
$\theta^\alpha_I$ are given by
\begin{gather}\label{VBdhdv}
d_H \theta_I^\alpha = dx^i \wedge \theta^\alpha_{Ii}, \quad d_V \theta^\alpha_I
= 0, \quad  d_H (dx^i) = 0, \quad d_V (dx^i) =0,
\\
d_H f = (D_i f) \, dx^i ,  \qquad   d_V f = \dfrac{\partial
f}{\partial u^\alpha} \theta^\alpha + \dfrac{\partial f}{\partial
u^\alpha_i} \theta^\alpha_i  + \dfrac{\partial f}{\partial
u^\alpha_{ij}} \, \theta^\alpha_{ij} +   \cdots \, , \notag
\end{gather}
where $D_i$ denotes total differentiation with respect to $x^i.$
The \emph{free variational bicomplex} is defined to be the double
complex $\{ \Omega^{r,s} J^\infty(E), \, d_H, \, d_V \, \}_{s
 \geq 0; \ r=0,1,\dots, n}.$

In our solution to the variational multiplier problem for the
$f$-Gordon systems
\begin{equation}\label{555gord}
u^\alpha_{xy} = f^\alpha(x,y,u^\gamma,u^\gamma_x,u^\gamma_y),
\qquad  \alpha =1,\dots,m,
\end{equation}
we investigate the existence of certain cohomology classes in the
\emph{constrained variational bicomplex} associated to a system of
partial differential equations.   To construct the constrained
variational bicomplex associated to \eqref{555gord},  we begin
with a trivial bundle $\pi:{\mathbf R}^2 \times {\mathbf R}^{m}
\rightarrow {\mathbf R}^2$  and consider the second-order jet
bundle $J^2({\mathbf R}^2,{\mathbf R}^m)$ with coordinates given
by
\begin{equation*}
(x,y,u^\alpha,u^\alpha_x,u^\alpha_y,u^\alpha_{xx},u^\alpha_{xy},u^\alpha_{yy}),
\quad \alpha=1,\dots,m.
\end{equation*}
An $f$-Gordon system \eqref{555gord} defines a
$(5m+2)$-dimensional submanifold  ${\mathcal R}^2 \overset
{\iota}{\rightarrow} J^2(E)$ called the {\it equation manifold} of
\eqref{555gord}.  We define the {\it first prolongation} of
${\mathcal R}^2$ as the $7m+2$-dimensional submanifold ${\mathcal
R}^3 \overset{\iota}{\rightarrow} J^3(E)$ defined by
\eqref{555gord} and $u^\alpha_{xxy} = D_x f^\alpha$ and
$u^\alpha_{xxy} = D_y f^\alpha.$ Further differentiations of
\eqref{555gord} will yield submanifolds ${\mathcal R}^k
\overset{\iota}{\rightarrow} J^k(E).$ For convenience we define
${\mathcal R}^0 = E$ and ${\mathcal R}^1 = J^1(E).$  We define the
{\it infinite prolonged equation manifold} ${\mathcal R}^\infty$
to be the inverse limit of the sequence $\{ {\mathcal R}^k \, | \,
k=0,1,2,\dots \},$  along with the natural projections
$\pi^\infty_M: {\mathcal R}^\infty \rightarrow M$ and
$\pi^\infty_k: {\mathcal R}^\infty \rightarrow {\mathcal R}^k.$ We
remark that there is a unique map $\iota_{\infty}:{\mathcal
R}^\infty  \rightarrow J^\infty(E)$ that satisfies the commutative
diagram
\begin{equation*}
\begin{CD}
{\mathcal R}^\infty  @>{\iota_\infty}>> J^\infty(E)
\\
@V{\pi^\infty_k}VV    @VV{\pi^\infty_k}V
\\
{\mathcal R}^k @>{\iota}>> J^k(E).
\end{CD}
\end{equation*}
For an $f$-Gordon system \eqref{555gord},  coordinates on
${\mathcal R}^\infty$ are given by
\begin{equation*}
(x,y,u^\alpha,u^\alpha_x,u^\alpha_y,u^\alpha_{xx},u^\alpha_{yy},u^\alpha_{xxx},u^\alpha_{yyy},u^\alpha_{xxxx},u^\alpha_{yyyy},
\dots).
\end{equation*}

We define the contact ideal ${\mathcal C}({\mathcal R}^\infty)$ on
${\mathcal R}^\infty$ via the pullback of the contact ideal on
$J^\infty(E),$ that is ${\mathcal C}({\mathcal R}^\infty) =
\iota^* {\mathcal C} (J^\infty(E)).$ The contact ideal ${\mathcal
C}({\mathcal R}^\infty)$ for an $f$-Gordon system \eqref{555gord}
is generated by the 1-forms
$\{\theta,\theta^\alpha_x,\theta^\alpha_y,\theta^\alpha_{xx},\theta^\alpha_{yy},
\theta^\alpha_{xxx},\theta^\alpha_{yyy}, \dots\},$ where
\begin{gather*}
\theta^\alpha= du^\alpha - u^\alpha_x dx - u^\alpha_y dy, \ \
\theta^\alpha_x = du^\alpha_x - u^\alpha_{xx} dx - f^\alpha \, dy,
 \ \theta^\alpha_y = du^\alpha_y - f^\alpha \, dx - u^\alpha_{yy}
dy,
\\
\theta^\alpha_{xx} = du^\alpha_{xx} - u^\alpha_{xxx} dx - \iota^*(D_x f^\alpha)
dy, \quad \theta^\alpha_{yy} = du^\alpha_{yy} - \iota^*(D_y f^\alpha) dx -
u^\alpha_{yyy} dy,
 \dots \ .
\end{gather*}
A basis for the $C^\infty(\eqman)$-module of 1-forms
$\Omega^1(\eqman)$ is given by
\begin{equation}\label{eqmancoframe}
\{ dx, dy, \theta^\alpha, \theta^\alpha_x,
\theta^\alpha_y,\theta^\alpha_{xx},\theta_{yy},
\theta^\alpha_{xxx},\theta^\alpha_{yyy}, \dots \}
\end{equation}

For every $p,$ we have a bi-grading $\Omega^p({\mathcal
R}^\infty)= \bigoplus \Omega^{r,s}({\mathcal R}^\infty),$ where
$r+s=p$ and $\Omega^{r,s}({\mathcal R}^\infty) = \iota_{\infty}^*
\Omega^{r,s} (J^\infty(E)).$  The exterior derivative
$d:\Omega^p({\mathcal R}^\infty) \to \Omega^{p+1} ({\mathcal
R}^\infty)$ splits as $d = d_H + d_V,$ where $d_H:\Omega^{r,s} \to
\Omega^{r+1,s}$ and $d_V: \Omega^{r,s} \to \Omega^{r,s+1}$ are the
{\it horizontal} and {\it vertical} differentials, respectively.
In the following section we will frequently use the $d_H$ and
$d_V$ structure equations for the coframe \eqref{eqmancoframe},
which are given by $d_H \, dx = d_H \, dy = d_V \, dx = d_V \,
dy=0,$
\begin{gather}\label{destructo}
d_H \theta^\alpha = dx \wedge \theta^\alpha_x + dy \wedge \theta^\alpha_y,
\qquad d_H \theta^\alpha_x = dx \wedge \theta^\alpha_{xx} + dy \wedge d_V
f^\alpha,
\\[7pt]
d_H \theta^\alpha_y = dx \wedge d_V f^\alpha + dy \wedge \theta^\alpha_{yy},
\qquad  d_V \theta^\alpha=d_V \theta^\alpha_{x^i}=d_V \theta^\alpha_{y^j}=0,
\notag
\\
d_H \, g =  D_x g \, dx + D_y g \, dy, \qquad d_V \, g = \dfrac{\partial
g}{\partial u^\alpha} \theta^\alpha + \dfrac{\partial g}{\partial u^\alpha_x}
\theta^\alpha_x + \dfrac{\partial g}{\partial u^\alpha_y} \theta^\alpha_y +
\cdots, \notag
\end{gather}
where $g \in C^\infty(\eqman)$ and $D_x$ and $D_y$ denote total differentiation
constrained to the equation manifold ${\mathcal R}^\infty.$

\begin{defin}  The {\it constrained variational bicomplex} $\Omega^{*,*}
({\mathcal R}^\infty,d_H,d_V)$ associated to an $f$-Gordon system
is the pullback of the free variational bicomplex $\Omega^{*,*}
(J^\infty(E),d_H,d_V)$ by $\iota_\infty:{\mathcal R}^\infty \to
J^\infty(E).$

\begin{equation*}
\begin{CD}
@.   @.     @AA{d_V}A    @AA{d_V}A     @AA{d_V}A
\\
\phantom{a} @.  0 @>>>  \Omega^{0,2}({\mathcal R}^\infty)   @>{d_H}>>
\Omega^{1,2} ({\mathcal R}^\infty)  @>{d_H}>> \Omega^{2,2} ({\mathcal
R}^\infty)
\\
@.  @.    @AA{d_V}A    @AA{d_V}A     @AA{d_V}A
\\
\phantom{a} @.  0 @>>>  \Omega^{0,1} ({\mathcal R}^\infty)
@>{d_H}>> \Omega^{1,1} ({\mathcal R}^\infty) @>{d_H}>>
\Omega^{2,1} ({\mathcal R}^\infty)
\\
@.  @.  @AA{d_V}A   @AA{d_V}A    @AA{d_V}A
\\
0 @>>> \mathbf R @>>> \Omega^{0,0} ({\mathcal R}^\infty) @>{d_H}>>
\Omega^{1,0}({\mathcal R}^\infty) @>{d_H}>> \Omega^{2,0}
({\mathcal R}^\infty)
\end{CD}
\end{equation*}
\end{defin}

The columns of the variational bicomplex on ${\mathcal R}^\infty$
are locally exact, while the rows will not be exact in general. We
define the horizontal cohomology classes of the variational
bicomplex by
\begin{equation*}
H^{r,s} ({\mathcal R}^\infty) = \dfrac{ \{ \omega \in
\Omega^{r,s}({\mathcal R}^\infty) \, |  \, d_H  \omega  = 0 \} } {
\{ d_H \eta \, | \, \eta \in \Omega^{r-1,s}({\mathcal R}^\infty)
\} }, \quad  \ r =1,2, \ \, s \geq 0.
\end{equation*}
It is easy to see that $H^{r,s} ({\mathcal R}^\infty)$ is a vector
space over $\mathbf R$.   The cohomology classes in
$H^{r,s}(\eqman)$ have interesting interpretations.  For example,
each class $[\omega]=H^{1,0}({\mathcal R}^\infty)$ represents a
{\it classical conservation law}.   Indeed, if $\omega = M \, dx +
N \, dy,$  then $d_H \omega = 0$ if and only if $D_x M = D_y N$
when restricted to the equation manifold ${\mathcal R}^\infty$. We
refer to a cohomology class $[\omega ] \in H^{1,s}({\mathcal
R}^\infty)$ as a {\it type $(1,s)$ conservation law} or {\it
form-valued conservation law}.  In the following section, we will
show that the solution to the variational multiplier problem is
closely related to the existence of non-trivial classes $[\omega]
\in H^{1,2}({\mathcal R}^\infty).$   Since systems of the form
\eqref{555gord} are of Cauchy-Kovaleskaya type, it follows from a
general result of Vinogradov \cite{Vino1984} that the  horizontal
cohomology spaces $H^{0,s}({\mathcal R}^\infty)$ satisfy
\begin{equation}\label{zerocomo}
H^{0,0} (\eqman) = {\mathbf R}, \quad \ \ H^{0,s} ({\mathcal
R}^\infty)=0, \ \ \ s>0.
\end{equation}
We remark that \eqref{zerocomo} was also established in
\cite{mydiss} by constructing a coframe adapted to systems
\eqref{555gord}.

%
%

\section{Derivation of Necessary and Sufficient Conditions for the Existence of a Variational Multiplier}
Our first result states that the problem of determining all
Lagrangians and variational multipliers for an $f$-Gordon system
is equivalent to determining all $d$ closed forms $\omega \in
\Omega^{1,2}(\eqman)$ of a certain type.   We also give a
description of the general form of possible Lagrangians for an
$f$-Gordon system.   Finally, we show that a variational
multiplier has no one-jet dependence. Proposition \eqref{TomLager}
along with Corollaries \eqref{lagerdiver} and \eqref{CorrNormF}
will suffice to establish Proposition \eqref{LnormalF}, which was
stated without proof in Section 2.

\begin{theorem}\label{TomLager}  For a system of differential equations  $u^\alpha_{xy} = f^\alpha(x,y,u^\gamma,
u^\gamma_x, u^\gamma_y)$  the following statements are equivalent.

\medskip

\noindent{(i) There exists a type $(1,2)$ form
\begin{equation}\label{FormOfOmega}
\omega =( T_{\alpha\bet} \, dx + S_{\alpha\bet}  \, dy ) \wedge
\theta^\alpha \wedge \theta^\bet \,  + \, R_{\alpha\bet}  ( \,
\theta^\alpha \wedge \theta_x^\bet \wedge dx \, -  \, \theta^\bet
\wedge \theta_y^\alpha \wedge dy \, )
\end{equation}
such that $\, d \omega = 0$ on ${\mathcal R}^\infty.$}

\medskip

\noindent(ii)  There exists a first-order multiplier $M_{\alpha\bet}
(x,y,u^\gamma, u^\gamma_x, u^\gamma_y)$ and a first-order Lagrangian
$L(x,y,u^\gamma,u^\gamma_x,u^\gamma_y)$ such that  $E_\alpha(L) =
M_{\alpha\bet} (u^\bet_{xy} - f^\bet).$

\medskip

{\noindent(iii) There exists a multiplier $M_{\alpha\bet} =
M_{\alpha\bet}(x,y,u^\gamma)$ and a Lagrangian
\begin{equation*}
L = -R_{\alpha\bet}(x,y,u^\gamma) u^\alpha_x  u^\bet_y  +
Q_\alpha(x,y,u^\gamma) u^\alpha_x + P_\alpha(x,y,u^\gamma)
u^\alpha_y + N(x,y,u^\gamma)
\end{equation*}
such that $E_{\alpha} (L) = M_{\alpha\bet} \left( u^\bet_{xy} -
f^\bet \right).$}
\end{theorem}

\begin{proof} We first show that (i) implies (iii).
Suppose that $\omega$ is given by \eqref{FormOfOmega} and that
$d\omega=0.$ on $\eqman.$  It follows immediately  that $d_H
\omega = d_V \omega =0. $ A routine calculation using
\eqref{destructo} shows that if $d_V \omega =0,$ then
$R_{\alpha\bet}=R_{\alpha\bet} (x,y,u^\epsilon)$ and $\omega$ is
of the form
\begin{align*} \omega &=\left[  \frac{1}{2}\left( \dfrac{\partial
R_{\alpha\gamma}}{\partial u^\bet} - \dfrac{\partial
R_{\bet\gamma}}{\partial u^\alpha} \right)  u^\gamma_x  +
T^0_{\alpha\bet} (x,y,u^\epsilon)\right] \theta^\alpha \wedge
\theta^\bet \wedge dx + R_{\alpha\bet}  \theta^\alpha \wedge
\theta^\bet_x \wedge dx
\\
& \  + \left[\frac{1}{2}\left( \dfrac{\partial
R_{\gamma\bet}}{\partial u^\alpha} - \dfrac{\partial
R_{\gamma\alpha}}{\partial u^\bet} \right) u^\gamma_y  +
S^0_{\alpha\bet} (x,y,u^\epsilon)\right] \theta^\alpha \wedge
\theta^\bet \wedge dy - R_{\bet\alpha}  \theta^\alpha \wedge
\theta^\bet_y \wedge dy. \notag
\end{align*}
If we define $\rho_0 \in \Omega^{1,1} (\eqman)$ by
\begin{equation*}
\rho_0 = -R_{\alpha\bet} u^\bet_x \ \theta^\alpha \wedge dx +
R_{\bet\alpha}  u^\bet_y \ \theta^\alpha \wedge dy,
\end{equation*}
then using\eqref{destructo} we see that
\begin{equation*}
\omega - d_V \rho_0= T^0_{\alpha\bet} \, \theta^\alpha \wedge
\theta^\bet \wedge dx + S^0_{\alpha\bet} \, \theta^\alpha \wedge
\theta^\bet \wedge dx.
\end{equation*}
Since $d_V (\omega - d_V \rho_0)=0,$ there exists a form $\rho_1
\in \Omega^{1,1} (\eqman) $ such that $d_V \rho_1= \omega - d_V
\rho_0.$  Since the functions $S^0_{\alpha\beta}$ and
$T^0_{\alpha\beta}$ have no 1-jet dependence, we may choose
$\rho_1$ to be of the form
\begin{equation*}
\rho_1 = -P_{\alpha}(x,y,u^\epsilon)  \, \theta^\alpha \wedge dx +
Q_{\alpha}(x,y,u^\epsilon) \, \theta^\alpha \wedge dy.
\end{equation*}
We now define $\rho = \rho_0 + \rho_1$ so that $d_V \rho =
\omega.$   Since $d_V d_H = -d_V d_H,$  we have $d_V (d_H \rho) =
-d_H d_V \rho = -d_H \omega = 0$ on $\eqman.$  Therefore there is
a Lagrangian form $\lambda \in \Omega^{0,2} (\eqman )$ with $d_V
\lambda = d_H \rho$ on $\eqman.$  We will now show that $\lambda=
L \, dx \wedge dy,$ where $L$ is of the form given in statement
(iii).    Computing $d_H \rho$ on $J^\infty(E)$ yields
\begin{align*}
d_H \rho  =& \  -\left[\left( R_{\alpha\bet} + R_{\bet\alpha} \right)
u^\bet_{xy} + g_\alpha(x,y,u^\epsilon,u^\epsilon_x,u^\epsilon_y ) \right]
\theta^\alpha \wedge dx\wedge dy -
\\
&\left( R_{\alpha\bet} u^\bet_x + P_\alpha \right) \theta^\alpha_y \wedge dx
\wedge dy- \left(R_{\bet\alpha} u^\bet_y + Q_\alpha \right)\theta^\alpha_x
\wedge dx \wedge dy,
\end{align*}
where
\begin{gather}\label{ggalpha}
g_\alpha = \left(\frac{\partial R_{\alpha\bet} }{\partial u^\gamma} +
\frac{\partial R_{\gamma\alpha} }{\partial u^\bet} \right) u^\bet_x u^\gamma_y
+ \left( \frac{\partial R_{\alpha\bet}}{\partial y} + \frac{\partial
Q_\alpha}{\partial u^\bet} \right)  u^\bet_x
\\
+ \left( \frac{\partial R_{\bet\alpha}}{\partial x} + \frac{\partial
P_\alpha}{\partial u^\bet} \right) u^\bet_y + \frac{\partial P_\alpha}{\partial
y} + \frac{\partial Q_\alpha}{\partial x}. \notag
\end{gather}
When restricted to $\eqman,$
\begin{equation*}
-d_H \rho =  \left[ \left(  2 R_{(\alpha\bet)} f^\bet + g_\alpha \right)
\theta^\alpha + \left(R_{\alpha\bet} u^\bet_x + P_\alpha \right)
\theta^\alpha_y   + \left(R_{\bet\alpha} u^\bet_y + Q_\alpha
\right)\theta^\alpha_x   \right] \wedge dx\wedge dy,
\end{equation*}
where $R_{(\alpha\bet)} = (R_{\alpha\bet} + R_{\bet\alpha})/2$ and
$g_\alpha$ is given by \eqref{ggalpha}.  Since $\lambda = L \, dx
\wedge dy$ and $d_V\lambda=d_H \rho$ on ${\mathcal R}^\infty,$ we
see that $L$ must satisfy
\begin{equation*}
\dfrac{\partial L}{\partial u^\alpha} = -2 R_{(\alpha\bet)} f^\bet
- g_\alpha, \quad \dfrac{\partial L}{\partial
u^\alpha_x}=-R_{\bet\alpha} u^\bet_y - Q_\alpha, \quad
\dfrac{\partial L}{\partial u^\alpha_y}=-R_{\alpha\bet} u^\bet_x -
P_\alpha.
\end{equation*}
It follows that there exists a function $N(x,y,u^\epsilon)$ such
that
\begin{equation}\label{Lnormal}
L=-\left( R_{\alpha\bet} u^\bet_x u^\alpha_y + Q_\alpha u^\alpha_x
+ P_\alpha u^\alpha_y + N(x,y,u^\epsilon) \right).
\end{equation}
If we apply the Euler-Lagrange operator $E(\lambda)=E_\alpha(L) \,
\theta^\alpha \wedge dx \wedge dy$ on $J^\infty(E),$  then
\begin{equation*}
d_H \rho + E(\lambda) = d_V (\lambda).
\end{equation*}
Since $d_H \rho = d_V \lambda$ when restricted to the equation
manifold $\eqman,$ we deduce that implies $\iota^{*} E(\lambda) =
0.$ On the other hand, a direct computation of the Euler-Lagrange
equations for \eqref{Lnormal} gives us
\begin{equation*}
E_\alpha(L) = 2 R_{(\alpha\bet)} u^\bet_{xy} + g_\alpha +  \dfrac{\partial
L}{\partial u^\alpha} .
\end{equation*}
Since $\iota^{*} E_\alpha(L) =0$ for all $\alpha,$ we get $
g_\alpha + {\partial L}/{\partial u^\alpha}  = -2R_{(\alpha\bet)}
f^\bet,$ which implies
\begin{equation*}
E_\alpha(L) = 2R_{(\alpha\bet)} ( u^\bet_{xy} - f^\bet ),
\end{equation*}
and (iii) is proved.   Moreover, the multiplier $M_{\alpha\bet}$
is defined by $M_{\alpha\bet}= R_{\alpha\bet} + R_{\bet\alpha}.$

We now prove that statement (iii) implies (i).  Assume that there
is  a variational multiplier $M_{\alpha\bet} (x,y,u^\gamma),$ a
first-order Lagrangian $\lambda = L \, dx \wedge dy,$ where $L$ is
of the form \eqref{Lnormal}, and
\begin{equation}\label{multisR}
E_\alpha(L) = M_{\alpha\bet} (u^\bet_{xy} - f^\bet).
\end{equation}
Note that  \eqref{multisR} explicitly determines that the
multiplier is $M_{\alpha\bet} = R_{\alpha\bet} + R_{\bet\alpha}.$
By the first variational formula (See \cite{And1992}, Corollary
5.3), we have $ E(\lambda) + d_H \eta = d_V \lambda,$ where
\begin{equation*}
\eta = \dfrac{\partial L}{\partial u^\alpha_y} \, \theta^\alpha
\wedge dx - \dfrac{\partial L}{\partial u^\alpha_x} \,
\theta^\alpha \wedge dy.
\end{equation*}
Since $\iota^{*} E(\lambda)=0,$  it follows that $\iota^{*} d_H
\eta = d_V \lambda.$  Then define $\omega = d_V \eta$ and the
resulting calculation of $d \omega$ on the equation manifold is
\begin{equation*} d \omega = d_H d_V \omega + d_V \omega =  d_H (d_V
\eta) + d_V d_V \eta = -d_V d_V \lambda= 0.
\end{equation*}
Moreover, using \eqref{destructo} to compute $d_V \eta$ will
verify that $\omega$ is of the form \eqref{FormOfOmega}.

We have  proven that (i) is equivalent to (iii), and clearly (iii)
implies (ii).    Assume that (ii) holds and there is a first-order
Lagrangian $L(x,y,u^\gamma, u^\gamma_x,u^\gamma_y)$ and a
variational multiplier $M_{\alpha\bet}(x,y,u^\gamma, u^\gamma_x,
u^\gamma_y)$ for the system $u^\bet_{xy} = f^\bet.$ We will show
that this implies (iii).   Calculating the Euler-Lagrange
equations for an arbitrary first-order Lagrangian, we find that
\begin{equation}\label{ladypilot}
E_\alpha(L)= - \dfrac{ \partial^2 L}{\partial u^\alpha_x \partial
u^\bet_x}  u^\bet_{xx} - \left( \dfrac{ \partial^2 L}{\partial
u^\alpha_x \partial u^\bet_y} + \dfrac{ \partial^2 L}{\partial
u^\alpha_y \partial u^\bet_x} \right) u^\bet_{xy}  -  \dfrac{
\partial^2 L}{\partial u^\alpha_y \partial u^\bet_y} u^\bet_{yy}
 +  G_\alpha,
\end{equation}
where $G_\alpha$ is a first-order function.   On the other hand, we are
assuming
\begin{equation}\label{deepred}
E_\alpha(L) = M_{\alpha\bet} \left( u^\bet_{xy}  -
f^\bet(x,y,u^\gamma,u^\gamma_x,u^\gamma_y) \right),
\end{equation}
where $M_{\alpha\bet}$ is a non-degenerate first-order multiplier.
Comparing \eqref{ladypilot} and \eqref{deepred}, we deduce that
\begin{equation}\label{nik1010}
\dfrac{ \partial^2 L}{\partial u^\alpha_x \partial u^\bet_x} = 0,
\qquad \dfrac{ \partial^2 L}{\partial u^\alpha_y
\partial u^\bet_y}=0.
\end{equation}
It follows that $L$ is of the form given in statement (iii) and
the fact that $M_{\alpha\bet}$ has no one-jet dependence follows
immediately from a calculation of the Euler-Lagrange equations for
a Lagrangian  satisfying \eqref{nik1010}.
\end{proof}

The following corollary states  that a Lagrangian $\lambda$
corresponding to a type (2,1) cohomology class $[\omega]$ is
unique up to modification by a total divergence.
\begin{corr}\label{lagerdiver}
Let $\omega$ be given by \eqref{FormOfOmega}  with  $d \omega= 0.$
If
\begin{equation*}
\lambda=L(x,y,u^\gamma,u^\gamma_x,u^\gamma_y) \, dx \wedge dy,
\qquad \lambda'=L'(x, y,u^\gamma,u^\gamma_x,u^\gamma_y) \, dx
\wedge dy
\end{equation*}
are two Lagrangians and $\eta, \,  \eta' \in \Omega^{1,1}(\eqman)$
satisfy
\begin{equation*}
d_H \eta=d_V \lambda, \qquad d_H \eta'=d_V \lambda', \qquad d_V \eta=d_V \eta'
= \omega,
\end{equation*}
then  $\lambda'=\lambda + d_H \beta$ for some form $\beta \in
\Omega^{1,0}(\eqman).$
\end{corr}

\begin{proof}
Since $d_V (\eta' - \eta) =0,$ there exists a form $\beta_0 \in
\Omega^{1,0}(\eqman)$ such that $\eta'=\eta + d_V \beta_0.$   It
follows that
\begin{equation*}
d_V \lambda'=d_H \eta'= d_H (\eta + d_V \beta_0) =  d_V \lambda +
d_H d_V \bet_0= d_V(\lambda - d_H \beta_0).
\end{equation*}
Consequently, $\lambda$ and $\lambda'$ satisfy
\begin{equation*}
d_V(\lambda' - \lambda + d_H \beta_0) = 0,
\end{equation*}
or equivalently that
\begin{equation*}
\lambda' - \lambda + d_H \beta_0 = a(x,y) \, dx \wedge dy.
\end{equation*}
Defining $\beta=\beta_0 - A(x,y) \, dy,$ where $A_x(x,y) =
a(x,y),$ we have that
\begin{equation*}
\eta' = \eta + d_V \beta \quad \text{and} \quad \lambda'=\lambda -
d_H \beta,
\end{equation*}
as required. \end{proof}

The following corollary gives a description the general form of a
$f$-Gordon systems that is variational.  Corollary
\eqref{CorrNormF}, along with Theorem \eqref{TomLager} and
Corollary \eqref{lagerdiver}, establishes Proposition
\eqref{LnormalF}.

\begin{corr}\label{CorrNormF} { If $u_{xy}^\bet = f^\bet(x,y,u^\gamma,
u^\gamma_x, u^\gamma_y)$ has a  nonsingular variational multiplier
$M_{\alpha\bet},$  then
\begin{equation*}
f^\bet = C^\bet_{\epsilon\tau} u^\epsilon_x u^\tau _y +
A^\bet_\epsilon u^\epsilon_x + B^\bet_\epsilon u^\epsilon_y +
G^\bet,
\end{equation*}
where $A^\bet_\epsilon, B^\bet_\epsilon, C^\bet_{\epsilon\tau},
G^\bet \in C^\infty(J^0(E)).$}
\end{corr}
\begin{proof}
According to Theorem \eqref{TomLager}, we may assume that there is
Lagrangian $L$ of the form \eqref{Lnormal} and a  variational
multiplier $M_{\alpha\bet}(x,y,u),$ with $\det M_{\alpha\bet} \neq
0$ and
\begin{equation}\label{ELclean}
E_\alpha(L) = M_{\alpha\bet} (u^\bet_{xy} - f^\bet).
\end{equation}
On the other hand,  the Euler-Lagrange equations for
\eqref{Lnormal} are explicitly given by
\begin{gather} \label{ELmess}
 E_\alpha(L) = (R_{\alpha\bet} + R_{\bet\alpha}) u^\bet_{xy} +
\left(\dfrac{\partial R_{\alpha\gamma}}{\partial u^\bet} + \dfrac{\partial
R_{\bet\alpha}} {\partial u^\gamma} -  \dfrac{\partial R_{\bet\gamma}}
{\partial u^\alpha}    \right) u^\bet_x u^\gamma_y \,  +
\\
 \left( \dfrac{\partial Q_\alpha}{\partial u^\bet} - \dfrac{\partial
Q_\bet}{\partial u^\alpha} + \dfrac{\partial
R_{\alpha\bet}}{\partial y} \right) u^\bet_x  +  \left(
\dfrac{\partial P_\alpha}{\partial u^\bet} - \dfrac{\partial
P_\alpha}{\partial u^\bet} + \dfrac{\partial
R_{\alpha\bet}}{\partial x} \right) u^\bet_y    -  \dfrac{\partial
N}{\partial u^\alpha} + \dfrac{\partial Q_\alpha}{\partial x}  +
\dfrac{\partial P_\alpha}{\partial y}. \notag
\end{gather}
Clearly,  $M_{\alpha\bet} = (R_{\alpha\bet} + R_{\bet\alpha})$ and
multiplying equations \eqref{ELclean} and \eqref{ELmess} by
$M^{\alpha\gamma},$ where $M^{\alpha\gamma} M_{\gamma\bet} =
\delta^\alpha_\bet,$ yields the desired result.
\end{proof}

In view of Theorem \eqref{TomLager}, we see that solving the
multiplier problem is equivalent to finding all type (1,2) forms
$\omega$ of the form \eqref{FormOfOmega} with $d\omega=0.$ From
Corollary \eqref{CorrNormF}, we can restrict ourselves to
examining systems of partial differential equations
\begin{equation}\label{sysnouxux}
u_{xy}^\alpha + C^\alpha_{\gamma\epsilon}(x,y,u^\bet) u^\gamma_x
u^\epsilon_y + A^\alpha_\gamma (x,y,u^\bet) u^\gamma_x +
B^\alpha_\gamma(x,y,u^\bet) u^\gamma_y + G^\alpha(x,y,u^\bet)=0.
\end{equation}
For systems \eqref{sysnouxux},  the following proposition gives
necessary and sufficient conditions for $d \omega =0.$

\begin{prop}\label{CondyMult}  There
exists a differential form $\omega \in \Omega^{1,2}(\eqman)$ of
type \eqref{FormOfOmega} with $d \omega =0$ if and only if
$M_{\alpha\bet} = \left(R_{\alpha\bet} + R_{\bet\alpha}\right)/2$
satisfies $M_{\alpha\bet}=M_{\alpha\bet}(x,y,u^\gamma)$ and
\begin{align}\label{MHK}
M_{\alpha\gamma} \, H^\gamma_\bet &=  M_{\bet\gamma} \,
K^\gamma_\alpha, \qquad
\\ \label{EqMS}
M_{\alpha\gamma} \, S^\gamma_{\bet\epsilon} &= - M_{\bet\gamma} \,
S^\gamma_{\alpha\epsilon} ,
\\ \label{dMFun}
d M_{\alpha\bet} &= M_{\alpha\gamma} \, \Omega^\gamma_\bet +
M_{\bet\gamma} \, \Omega^\gamma_\alpha,
\end{align}
where $S^\gamma_{\bet\epsilon} = C^\gamma_{\bet\epsilon} -
C^\gamma_{\epsilon\bet} \, $ and $ \, \Omega^\gamma_\alpha =
C^\gamma_{\bet\epsilon} \, du^\epsilon + B^\gamma_\alpha \, dx +
A^\gamma_\alpha \, dy.$
\end{prop}

\begin{proof} In the proof of Theorem \eqref{TomLager} we showed that if $d_V \omega=0,$
then $R_{\alpha\beta}$ has no 1-jet dependence.   Since
$M_{\alpha\bet} = (R_{\alpha\bet} + R_{\bet\alpha})/2,$ it follows
immediately that the functions $M_{\alpha\beta}$ depends only $x,$
$y$ and $u.$

To simplify our calculations, we define $\omega' = \omega -
\frac{1}{4} d_H \left(R_{\alpha\bet}-R_{\bet\alpha}\right)
\theta^\alpha \wedge \theta^\bet.$ A routine calculation using
\eqref{destructo} yields
\begin{equation}\label{FormOfOmegaPr}
\omega'= ( T_{\alpha\bet} dx + V_{\alpha\bet} dy ) \wedge
\theta^\alpha \wedge \theta^\bet  + M_{\alpha\bet} ( \theta^\alpha
\wedge \theta_x^\bet  \wedge dx - \theta^\bet \wedge
\theta_y^\alpha \wedge dy),
\end{equation}
where $M_{\alpha\bet} = (R_{\alpha\bet} + R_{\bet\alpha})/2.$
Without loss of generality we may assume that $T_{\alpha\beta}$
and $V_{\alpha\bet}$ are skew-symmetric. Using \eqref{zerocomo}
and the exactness of the columns of the variational bicomplex, it
can be shown  that $d \omega=0$ if and only if $d_H \omega'=0$ and
$d_V \omega'$ is $d_H$ exact.   We will later show that if $d_H
\omega'=0$ for \eqref{FormOfOmegaPr}, then $d_V \omega'$ is
necessarily $d_H$ exact.

We first calculate $d_H \omega'$ and show that $d_H \omega'=0$ if
and only if \eqref{MHK}-\eqref{dMFun} hold.   Using
\eqref{destructo} we calculate
\begin{align}\label{old55}
d_H \omega' &= \left[\left(D_x V_{\alpha\bet}  - D_y T_{\alpha\bet} \right)
\theta^\alpha \wedge \theta^\bet  - 2 M_{\alpha\bet} \theta^\alpha\wedge (d_V
f^\bet) \right] \wedge dx \wedge dy \, -
\\
& \ \  \left[\left( D_y M_{\alpha\bet} - 2V_{\alpha\bet} \right) \theta^\alpha
\wedge \theta_x^\bet  + \left( D_x M_{\alpha\bet} + 2T_{\alpha\bet} \right)
\theta^\alpha \wedge \theta_y^\bet \right]  \wedge dx \wedge dy. \notag
\end{align}
It follows that $d_H \omega'=0$ if and only if
\begin{gather}\label{681eq}
D_x M_{\alpha\bet}  + 2 T_{\alpha\bet} + 2 M_{\alpha\gamma} \dfrac{\partial
f^\gamma}{\partial u^\bet_y} = 0, \qquad  D_y M_{\alpha\bet} - 2 V_{\alpha\bet}
+ 2 M_{\alpha\gamma} \dfrac{\partial f^\gamma}{\partial u^\bet_x} = 0,
\\\label{680eq}
D_x V_{\alpha\bet} - D_y T_{\alpha\bet} - M_{\alpha\gamma} \dfrac{\partial
f^\gamma}{\partial u^\bet} + M_{\bet\gamma} \dfrac{\partial f^\gamma}{\partial
u^\alpha}=0.
\end{gather}
Decomposing the two equations \eqref{681eq} into their symmetric
and skew symmetric parts produces four equations
\begin{gather}\label{682eq}
D_x M_{\alpha\bet} = -M_{\alpha\gamma} \dfrac{\partial f^\gamma}{\partial
u^\bet_y} - M_{\bet\gamma} \dfrac{\partial f^\gamma}{\partial u^\alpha_y},
\quad D_y M_{\alpha\bet} = -M_{\alpha\gamma} \dfrac{\partial f^\gamma}{\partial
u^\bet_x} - M_{\bet\gamma} \dfrac{\partial f^\gamma}{\partial u^\alpha_x},
\\ \label{683eq}
T_{\alpha\bet} = \dfrac{1}{2} \left( M_{\bet\gamma} \dfrac{\partial
f^\gamma}{\partial u^\alpha_y} - M_{\alpha\gamma} \dfrac{\partial
f^\gamma}{\partial u^\bet_y} \right), \qquad V_{\alpha\bet} = \dfrac{1}{2}
\left( M_{\alpha\gamma} \dfrac{\partial f^\gamma}{\partial u^\bet_x} -
M_{\bet\gamma} \dfrac{\partial f^\gamma}{\partial u^\alpha_x} \right).
\end{gather}
Substituting \eqref{683eq} into \eqref{680eq} and taking
\eqref{682eq} into account, we arrive at
\begin{equation}\label{MHPLUSK}
M_{\alpha\gamma}( H^\gamma_\bet + K^\gamma_\bet ) = M_{\bet\gamma}
(H^\gamma_\alpha + K^\gamma_\alpha).
\end{equation}
where
\begin{equation*}
H_{\alpha}^\gamma = \dfrac{\partial f^\gamma}{\partial u^\alpha} +
\dfrac{\partial f^\gamma}{\partial u^\tau_y} \dfrac{\partial
f^\tau}{\partial u^\alpha_x} - D_x \left[ \dfrac{\partial
f^\gamma}{\partial u^\alpha_x} \right], \quad K_{\alpha}^\gamma =
\dfrac{\partial f^\gamma}{\partial u^\alpha} + \dfrac{\partial
f^\gamma}{\partial u^\tau_x} \dfrac{\partial f^\tau}{\partial
u^\alpha_y} - D_y \left[ \dfrac{\partial f^\gamma}{\partial
u^\alpha_y} \right].
\end{equation*}
It follows that $d_H \omega'=0$ if and only if \eqref{682eq} and
\eqref{MHPLUSK} hold. For a system of differential equations
\eqref{sysnouxux},  we expand \eqref{682eq} to arrive at
\begin{align*}
\dfrac{\partial M_{\alpha\bet}}{\partial x \ \ }  +\dfrac{\partial
M_{\alpha\bet}}{\partial u^\epsilon \ } u^\epsilon_x &=
M_{\alpha\gamma} \left(C^\gamma_{\epsilon\bet} u^\epsilon_x +
B^\gamma_\bet \right) + M_{\bet\gamma}
\left(B^\gamma_{\epsilon\alpha} u^\epsilon_x +
A^\gamma_\alpha\right),
\\
\dfrac{\partial M_{\alpha\bet}}{\partial y \ \ }  +\dfrac{\partial
M_{\alpha\bet}}{\partial u^\epsilon \ } u^\epsilon_y &= M_{\alpha\gamma}
\left(C^\gamma_{\bet\epsilon} u^\epsilon_y  + A^\gamma_\bet \right) +
M_{\bet\gamma} \left(C^\gamma_{\alpha\epsilon} u^\epsilon_y +
A^\gamma_\alpha\right).
\end{align*}
Since $M_{\alpha\bet}$ has no first-order dependence, we have
\begin{gather}\label{goodburger}
\dfrac{\partial M_{\alpha\bet}}{\partial u^\epsilon \ } = M_{\alpha\gamma}
C^\gamma_{\bet\epsilon} + M_{\bet\gamma} C^\gamma_{\alpha\epsilon}, \qquad
\dfrac{\partial M_{\alpha\bet}}{\partial u^\epsilon \ } = M_{\alpha\gamma}
C^\gamma_{\epsilon\bet} + M_{\bet\gamma} C^\gamma_{\epsilon\alpha},
\\
\dfrac{\partial M_{\alpha\bet}}{\partial x \ \ } =
M_{\alpha\gamma} B^\gamma_{\bet} + M_{\bet\gamma}
B^\gamma_{\alpha}, \qquad  \dfrac{\partial
M_{\alpha\bet}}{\partial y \ \ } = M_{\alpha\gamma}
A^\gamma_{\bet} + M_{\bet\gamma} A^\gamma_{\alpha}.  \notag
\end{gather}
It follows that \eqref{goodburger} is equivalent to conditions
\eqref{EqMS} and \eqref{dMFun}.  If \eqref{EqMS} holds, then it
can be shown that the integrability conditions $d^2
M_{\alpha\bet}=0$ for \eqref{dMFun} are
\begin{equation}\label{pizza}
M_{\alpha\sigma} \left( H^\sigma_\bet - K^\sigma_\bet \right) +
M_{\bet\sigma} \left( H^\sigma_\alpha - K^\sigma_\alpha \right) =
0.
\end{equation}
Equations  \eqref{pizza} and \eqref{MHPLUSK}  both hold if and
only if $M_{\alpha\sigma} H^\sigma_\bet = M_{\bet\sigma}
K^\sigma_\alpha,$ which is precisely condition \eqref{MHK}.

We now show that $d_V \omega'$ is $d_H$ exact whenever $d_H
\omega'=0.$  We claim that $d_H \zeta = d_V \omega'$ for
\begin{equation}
\zeta = \tfrac{1}{6} M_{\alpha\sigma} S^\sigma_{\bet\gamma} \,
\theta^\alpha \wedge \theta^\bet \wedge \theta^\gamma
\end{equation}
From \eqref{sysnouxux}, \eqref{FormOfOmegaPr}, and \eqref{683eq},
we deduce that $\omega'$ can be expressed as
\begin{gather*}
\omega'=M_{\alpha\sigma} \, \left( C^\sigma_{\tau\bet} u^\tau_x +
B^\sigma_\bet \right) \, \theta^\alpha \wedge \theta^\bet \wedge
dx    +  M_{\alpha\bet} \, \theta^\alpha \wedge \theta^\bet_x
\wedge dx  \, -
\\
\quad M_{\alpha\sigma} \left(  C^\sigma_{\bet\tau} u^\tau_y +
A^\sigma_\bet \right) \, \theta^\alpha \wedge \theta^\beta \wedge
dy  - M_{\alpha\bet}  \theta^\alpha \wedge \theta^\bet_y \wedge
dy.
\end{gather*}
We now calculate $d_V \omega' - d_H \zeta.$  Since we are assuming
that $d_H \omega'=0,$  we use  \eqref{EqMS}, \eqref{dMFun}, and
the fact that $S^\sigma_{\bet\gamma}= C^\sigma_{\bet\gamma} -
C^\sigma_{\gamma\bet}$  to arrive at
\begin{gather} \label{6111nik}
d_V  \omega' -  d_H \zeta = Q_{\alpha\beta\gamma} \, \theta^\alpha
\wedge \theta^\bet \wedge \theta^\gamma \wedge dx \ + \
R_{\alpha\beta\gamma} \, \theta^\alpha \wedge \theta^\bet \wedge
\theta^\gamma \wedge dy \ +
\\
\left[\left( \tfrac{1}{2} M_{\alpha\sigma} S^\sigma_{\bet\gamma} -
M_{\gamma\sigma} C^\sigma_{\alpha\bet} \right) \theta^\gamma_x \wedge dx \, -
\left( \tfrac{1}{2} M_{\alpha\sigma} S^\sigma_{\bet\gamma} - M_{\gamma\sigma}
C^\sigma_{\alpha\bet}  \right) \theta^\gamma_y \wedge dy \right] \wedge
\theta^\alpha \wedge \theta^\bet, \notag
\end{gather}
where $Q_{\alpha\beta\gamma}$ and $R_{\alpha\beta\gamma}$ are
totally skew-symmetrized expressions depending on
$M_{\alpha\bet},$ $C_{\bet\gamma}^\alpha,$ $A^\alpha_\gamma,$
$B^\alpha_\gamma$ and their derivatives.    If we skew-symmetrize
over $\alpha$ and $\beta,$ it follows from \eqref{EqMS} that
\eqref{6111nik} simplifies to
\begin{equation}\label{almostlastcheck}
d_V \omega' - d_H \zeta =  Q_{\alpha\bet\gamma} \theta^\alpha \wedge
\theta^\bet \wedge \theta^\gamma \wedge dx + R_{\alpha\bet\gamma} \theta^\alpha
\wedge \theta^\bet \wedge \theta^\gamma \wedge dy
\end{equation}
Since $d_H (d_V \omega' - d_H \zeta) = -d_V(d_H \omega')=0,$
applying $d_H$ to the right hand side of \eqref{almostlastcheck}
produces the equation
\begin{equation}\label{qqzero}
\left[ 3Q_{\alpha\bet\gamma} \theta^\alpha_y  -  3 R_{\alpha\bet\gamma}
\theta^\alpha_x  +  (D_y Q_{\alpha\beta\gamma}  - D_x R_{\alpha\bet\gamma})
\theta^\alpha \, \right] \wedge \theta^\bet \wedge \theta^\gamma \wedge dx
\wedge dy = 0.
\end{equation}
It follows form \eqref{qqzero} that
$Q_{\alpha\bet\gamma}=R_{\alpha\beta\gamma} = 0$ and we can
conclude that $d_V \omega'=d_H \zeta.$
\end{proof}

\section{Concluding Remarks}

We conjecture that the methods used in this paper can be used to obtain a complete solution of the inverse problem for systems of the form
\begin{equation}\label{fart2009}
g^{ij}(x,y,u^\gamma) h_{\alpha\beta} u^\beta_{ij}  + f_\alpha(x^k,u^\gamma,u^\gamma_k)=0,
\end{equation}
where $\alpha,\beta=1,\dots,m$ and $i,j=1,\dots,n.$  In particular, if $\det g^{ij} \neq 0$ and $\det h_{\alpha\beta}\neq 0,$ then the system \eqref{fart2009} is a generalization of harmonic map equation.   A partial solution of the inverse problem harmonic map equation was obtained was obtained by Henneaux \cite{Henneaux1984}.   It would be some interest to derive a set of invariant analogous to the invariant derived in this paper that completely characterize the existence of Lagrangians for systems \eqref{fart2009}.

\setcounter{section}{0} \setcounter{theorem}{0}
\renewcommand{\thetheorem}{\thesection.\arabic{theorem}}
\setcounter{equation}{0}
\renewcommand{\theequation}{\thesection.\arabic{equation}}
\setcounter{figure}{0} \setcounter{table}{0}

\appendix
\section*{Appendix}
\renewcommand{\thesection}{A}

To complete our solution to the inverse problem, we have the
following general result on existence of solutions to systems of
total differential equations with an algebraic constraint.
Proposition \eqref{totallyawesome} establishes \eqref{Epop1} and
\eqref{Epop2} and completes the proof of Theorem
\eqref{inversesolution}.

\begin{prop}\label{totallyawesome} {Consider the system of total differential
equations
\begin{equation}\label{bigde1} \dfrac{\partial z^\alpha}{\partial
x^i} = A^\alpha_{\gamma i } (x) z^\gamma \qquad  1 \leq \alpha \leq m ; \ \ 1
\leq i \leq n,
\end{equation}
coupled with the algebraic  condition}
\begin{equation}\label{sidecond}
B^a_\gamma (x) z^\gamma = 0,   \qquad 1\leq a \leq l.
\end{equation}
{Suppose the integrability conditions for \eqref{bigde1} are
satisfied whenever \eqref{sidecond} holds and that given $x_0 \in
{\mathbf R}^n,$ there exists a neighborhood $U$ of $x_0$ such that
$ \text{Rank} \left\{ B^a_\gamma \left(x_0 \right)\right\} =
\text{Rank} \left\{ B^a_\gamma \left(x\right)\right\}=k \leq l,$
for all  $\ x \in U.$ We also assume that the algebraic system
\eqref{sidecond} is complete in the sense that differentiating
\eqref{sidecond} and substituting from \eqref{bigde1} produces no
new algebraic conditions.  Given $x_0 \in U \subset {\mathbf R}^n$
and $z_0 \in {\mathbf R}^m$ with $B^a_\gamma (x_0) z_0^\gamma=0,$
there exists unique smooth functions $z^\alpha = z^\alpha(x)$
defined on a neighborhood  $V \subset U$ of $x_0,$ such that
$z^\alpha(x)$ satisfy both \eqref{bigde1} and \eqref{sidecond},
and $z^\alpha(x_0) = z^\alpha_0.$}
\end{prop}

\begin{remark} For the more geometrically inclined reader we note
that Proposition \eqref{totallyawesome} is a special case of the
following result: {For an $m+n$ dimensional manifold $M$ let
${\mathcal I} \subset \Omega^1(M)$ denote an exterior differential
system of rank $m$ and let $q$ be a regular value of a smooth map
$F:M \to Q.$  Assume for all $p \in \Sigma = F^{-1}(q)$ we have
${\mathcal I}^\bot_p \subset \text{\rm Ker} \, (F_*:T_p M \to T_q
Q)$ and that ${\mathcal I}_{{}_{|\Sigma}}$ is a Frobenius system.
Then for all $p_0 \in \Sigma \subset M$ there is a unique maximal
$n$-dimensional integral manifold $\phi:N \to M$ through $p_0$
such that $\phi(N) \subset \Sigma.$}
\end{remark}

\begin{proof}  Since we are only constructing local solutions to
\eqref{bigde1} and \eqref{sidecond}, we will assume that
$\text{Rank} \{ B^a_\alpha(x) \} = k$ for all $x \in {\mathbf
R}^n.$ Assume that no additional algebraic constraints are created
if we differentiate \eqref{sidecond} with respect to $x^i$ and
substitute from \eqref{bigde1}.   This guarantees the existence of
functions $G^a_{ci}(x)$ such that
\begin{equation}\label{nonewcond}
\dfrac{\partial B^a_\gamma}{\partial x^i} + B^a_\sigma \, A^\sigma_{\gamma i} =
G^a_{ci} \, B^c_\gamma.
\end{equation}

The system of differential equations \eqref{bigde1} are associated
with the $C^\infty$ distribution $\Delta$ on ${\mathbf R^n} \times
{\mathbf R^m}$ generated by the vector fields
\begin{equation}
X_i  = \dfrac{\partial}{\partial x^i}  + A^\gamma_{\sigma i} z^\sigma \,
\dfrac{\partial}{\partial z^\gamma}, \quad i=1,2, \dots, n.
\end{equation}
We then define the $C^\infty$ function $F:{\mathbf R}^n \times {\mathbf R}^m
\to {\mathbf R}^l$ by $F(x,z)=B^a_\gamma (x) z^\gamma$ and let $\Sigma =
F^{-1}(0).$    As a consequence of \eqref{nonewcond} we see that $\text{Rank}
\{ DF_{(x,z)} \} = \text{Rank} \, \{B^a_\gamma(x)\}=k$ for all $(x,z)\in
\Sigma.$   It follows immediately that $\Sigma \subset {\mathbf R}^n \times
{\mathbf R^m}$ is regular submanifold of dimension $m+n-k.$   In addition, we
can use \eqref{nonewcond} to show that $X_i( B^a_\gamma \, z^\gamma)$ vanishes
identically for all $(x,z) \in \Sigma.$ We conclude that for any $p \in \Sigma$
and $i=1,2,\dots,m,$
\begin{equation}\label{tspacesigma}
X_{i_{|\text{\scriptsize$p$}}} \in \text{\rm Ker} ( F_*: T_p {\mathbf R}^{n+m}
\to T_0 {\mathbf R}^l) = T_p \Sigma.
\end{equation}
It follows from \eqref{tspacesigma} that $\Delta_{|\Sigma}$ is a
$C^\infty$ distribution on the submanifold $\Sigma.$  (See Boothby
\cite{BoothbyBook}, Lemma IV.2.4.)  If the integrability
conditions are satisfied for \eqref{bigde1} whenever
\eqref{sidecond} holds, then $\Delta_{|\Sigma}$ is involutive. For
any point $(x_0,z_0) \in \Sigma,$ we can invoke the Frobenius
Theorem to obtain a unique maximal integral manifold $\phi:V \to
\Sigma$ with $V \subset {\mathbf R}^n,$ $\phi(0)=(x_0,z_0),$  and
$T_p \, \phi(V) = \Delta_p.$ Moreover, $\phi$ is $C^\infty$ viewed
as map $\phi:V \to {\mathbf R}^{n+m}$ and $\phi(V)$ defines the
graph of a smooth solution $z^\alpha(x)$ to the system of
differential equations \eqref{bigde1}. Since $\phi(V) \subset
\Sigma,$ we see that $z^\alpha(x)$ satisfies the algebraic
condition \eqref{sidecond}.
\end{proof}

\end{document}